\documentclass[11pt]{article}

\usepackage[english]{babel}
\usepackage[utf8]{inputenc}
\usepackage{geometry}
\usepackage{amsmath, amssymb, amsthm}
\usepackage{graphicx}
\usepackage[colorlinks=true, allcolors=blue]{hyperref}
\usepackage{enumitem}
\usepackage{mathtools}
\usepackage[capitalise]{cleveref}
\usepackage{fullpage}

\usepackage{tikz, enumitem}
\usetikzlibrary{shapes.geometric, calc}

\usetikzlibrary{shapes.geometric, calc}

\newtheorem{theorem}{Theorem}
\newtheorem{lemma}[theorem]{Lemma}
\newtheorem{fact}[theorem]{Fact}

\newtheorem{corollary}[theorem]{Corollary}

\newtheorem{problem}[theorem]{Problem}
\newtheorem{proposition}[theorem]{Proposition}

\theoremstyle{definition}

\theoremstyle{remark}
\newtheorem{remark}[theorem]{Remark}

\renewcommand{\epsilon}{\varepsilon}
\newcommand{\eps}{\epsilon}

\newcommand{\Poiss}{\operatorname{Poiss}}
\newcommand{\TV}{\mathrm{TV}}
\newcommand{\cZ}{\mathcal{Z}}

\newcommand{\cY}{\mathcal{Y}}
\newcommand{\EE}{\mathbb{E}}
\newcommand{\E}{\mathbb{E}}

\newcommand{\cF}{\mathcal{F}}
\newcommand{\cP}{\mathcal{P}}
\newcommand{\cJ}{\mathcal{J}}

\newcommand{\PP}{\mathbb{P}}

\newcommand{\cG}{\mathcal{G}}

\newcommand{\cA}{\mathcal{A}}
\newcommand{\cC}{\mathcal{C}}
\newcommand{\cD}{\mathcal{D}}
\newcommand{\cE}{\mathcal{E}}

\newcommand{\cB}{\mathcal{B}}

\newcommand{\cI}{\mathcal{I}}

\title{The largest $K_r$-free set of vertices in a random graph}
\author{
Tom Bohman\thanks{Department of Mathematical Science, Carnegie Mellon University. Email: tbohman@math.cmu.edu. Research partially supported by NSF Award DMS-2246907} \and  
Marcus Michelen\thanks{Department of Mathematics, Northwestern University. Email: michelen@northwestern.edu. Research partially supported by NSF awards DMS-2336788 and DMS-2246624} \and  
Dhruv Mubayi\thanks{Department of Mathematics, Statistics and Computer Science, University of Illinois, Chicago, IL 60607. Email: mubayi@uic.edu. Research partially supported by NSF Award DMS-2153576.}
}
\date{}

\begin{document}

\maketitle

\begin{abstract}
For $r \ge 2$ and a graph $G$, let $\alpha_{{r}}(G)$ be the maximum number of vertices in a $K_r$-free subgraph  of $G$. We investigate the value $\alpha_{r}(G)$ when $G$ is the random graph $G \sim G_{n, 1/2}$ and discover the following 
phenomenon: with high probability, $\alpha_r(G)$ lies in an interval of  constant length that varies in a non-monotonic fashion from $1$ to $\lfloor r/2\rfloor+1$ depending on the value of $n$.   The special case $r=2$ corresponds to the independence number of random graphs which is well-known to have two-point concentration; our results therefore extend and generalize this basic fact in random graph theory, showing more complicated behavior when $r>2$. We also prove similar results where $K_r$ is replaced by any color critical graph like $C_5$.
\end{abstract}

\section{Introduction}

Consider the binomial random graph $G_{n,1/2}$ and let $\alpha(G_{n,1/2})$ be its \emph{independence number}, i.e.\ the size of the largest independent set of $G$.  In the 1970's, classical works of Bollob\'as and Erd\H{o}s \cite{BE} and Matula \cite{Matula}  independently showed that $\alpha(G_{n,1/2})$ is concentrated on at most two values meaning that there is a deterministic function $f(n)$ so that $$\lim_{n \to \infty}\PP(\alpha(G_{n,1/2}) \in \{f(n),f(n)+1\}) = 1\,.$$
In fact, the works of Bollob\'as-Erd\H{o}s and Matula \cite{BE, Matula} imply that for a typical integer $n$, $\alpha(G_{n,1/2})$ is concentrated on a single value, meaning that for most $n$ we have $\alpha(G_{n,1/2}) = f(n)$ with high probability. Similar results are also known for $G \sim G_{n, p}$ as long as $p$ is not sufficiently small in terms of $n$~\cite{F, DM, BH, BH2}.  One may interpret an independent set as simply a subset of $G_{n,1/2}$ that avoids the presence of an edge, i.e.\ a clique on two vertices.  With this in mind, for a graph $G$, define $\alpha_{r}(G)$ to be the maximum number of vertices in an induced subgraph of $G$ that does not contain a copy of $K_r$, the clique on $r$ vertices.  We note in passing that instead of vertices one can also consider the maximum number of edges in a subgraph of $G_{n,p}$ that contains no copy of $K_{r}$; this 
problem, posed by Babai-Simonovits-Spencer~\cite{BSS} in 1990, has a long and storied history. Some of the  influential results here are due to Conlon-Gowers~\cite{CG}, Schacht~\cite{Schacht}, and DeMarco-Kahn~\cite{DK}.  In analogy with the independence number, we study the problem of concentration of the random quantity $\alpha_{r}(G_{n,1/2})$.  In particular, does two-point concentration still hold for $\alpha_r(G_{n,1/2})$?

Perhaps surprisingly, one consequence of our work is that two-point concentration holds for $r = 2,3$ and fails for $r \geq 4$.  We show that for each fixed\footnote{Our main result  \cref{thm:poisson-version} will in fact apply for $r =  o(\log \log n/\log\log\log n)$.} $r \geq 2$ we have that $\alpha_r(G_{n,1/2})$ lies in an interval of constant length that varies in a non-monotonic fashion from $1$ to $\lfloor r/2\rfloor + 1$ depending on the specific value of $n$. This phenomenon is closely tied to the fact that as $n$ varies, the number of independent sets of size $\alpha(G_{n,1/2})$ can be as small as $O(1)$ or as large as $n^{1 + o(1)}$ depending on $n$.  This non-monotonicity is also what drives the changing exponent of the concentration window for the chromatic number of $G_{n,1/2}$ in the work of Heckel \cite{heckel2021non} (see also Heckel-Riordan \cite{heckel2023does} for a precise conjecture).

As is the case with the independence number, for a typical value of $n$, $\alpha_r(G_{n,1/2})$ is concentrated on a single value; when not concentrated on a single value, the size of the interval can take on multiple values for a given $r$ that depend on divisibility properties  of $r$. 
For instance, in the case of, say, $r = 1001$, there will be values of $n$ for which $\alpha_r(G_{n,1/2})$ is concentrated on exactly $2$ values, some where it is concentrated on exactly $501$ values, some where it is concentrated on exactly $84$ values.  In fact, we will show that in the case of $r = 1001$, as $n$ varies $\alpha_r(G_{n,1/2})$ is concentrated on intervals of length $\{1, 2, 3, 4, 5, 6, 8, 11, 12, 15, 18, 25, 35, 51, 84, 168, 501\}$.  The pattern of what intervals of concentration occurs is non-monotone, with always going to $1$ before a larger value.  In particular, we first get $1$ point concentration, then $2$ point concentration; it alternates between $1$-point and $2$-point concentration $34$ times before seeing $3$-point.  Further, the interval can shrink even aside from going back to $1$ point.  In the case of $r =1001$ we have a stretch where we see $3$-point then $1$-point, $2$-point, $1$-point, $3$-point, $1$-point, $2$-point. In general the pattern is complicated, but is fully described by \cref{thm:clique}.

Our most general main theorem is \cref{thm:poisson-version} which describes a criterion for witnessing both upper and lower bounds of $\alpha_r(G_{n,1/2})$ along with a statement identifying Poissonian behavior for these witnesses; we use \cref{thm:poisson-version} to deduce a complete description of the intervals of concentration in \cref{thm:clique}.  Before jumping in to describe this general case, we describe the story for the case of $r = 3$, i.e.\ the largest triangle-free subgraph of $G_{n,1/2}$.  We describe the case of triangles along with a heuristic description of what happens for larger cliques in \cref{sec:triangles}, state our more general theorem in \cref{sec:Poissonian} and describe the intervals of concentration in \cref{sec:intervals}.  Finally, we describe some slightly coarser results for when avoiding \emph{color critical} graphs other than $K_r$ in \cref{sec:other-properties}. 

\subsection{The case of triangles and a heuristic description for larger cliques}\label{sec:triangles}

The expected number of triangle-free subgraphs of size $k$ in $G_{n,1/2}$ is precisely given by $N_3(k) \binom{n}{k}2^{-\binom{k}{2}}$ where $N_3(k)$ is the number of triangle-free graphs on vertex set $[k]$.  A classical theorem of Erd\H{o}s-Kleitman-Rothschild \cite{EKR} states that almost all triangle-free graphs are bipartite graphs.  With this in mind, one might hope that rather than searching for a large triangle-free graph in $G_{n,1/2}$, we might be able to simply search for a large bipartite graph.  Interestingly, it turns out that this is not precisely the case.

Suppose we are in the case where for some value of $n$, the independence number is $k$ with high probability.  As $n$ increases, the independence number will eventually be $k+1$ with high probability.  As this transition occurs, before seeing genuine independent sets of size $k+1$, we will begin to see some number of sets of $k+1$ containing only $1$ edge.  We call such an edge a \emph{defect}.  If one considers an independent set of size $k$ and a $(k+1)$-set with a defect, it will not be too common that the union of these two sets induces a triangle free graph---indeed one can compute that the probability a given $k$-set and $(k+1)$-set with a defect combine to give a triangle-free graph is $(3/4)^{-k}$---but despite this probability being so small, it turns out that by the time there is a $(k+1)$-set with a defect there are enough independent sets of size $k$ that in fact it is likely that a triangle-free graph of size $2k + 1$ occurs.  

In particular, if we define $Y_k$ to be the number of independent sets of size $k$ and $Z_{k,1}$ to be the number of $k$-sets of vertices that contain exactly $1$ edge, then \cref{thm:poisson-version} implies that \begin{gather*}
    \PP(\alpha_3(G_{n,1/2}) \geq 2k+1)  = \PP(Z_{k+1,1} \geq 1) + o(1) \\
    \PP(\alpha_3(G_{n,1/2}) \geq 2k)  = \PP(Y_{k} \geq 2) + o(1)
\end{gather*}
for each $k$ chosen so that $n^{-O(1)} \leq \E[Y_k] \leq n^{O(1)}.$  Further, each of $Z_{k,1}$ and $Y_k$ behave like Poisson random variables provided their mean is, say, no more than $n^{1/4}.$  To describe the intervals of concentration, for a small parameter $\eps = o(1)$ we may define \begin{gather*}
    a_k = \min\{ n : \E[Y_k] \geq \eps^{-1} \}\\
    b_{k,1} = \min\left\{ n : \EE  [Z_{k+1, 1}]\ge \epsilon \right\}\,, \qquad 
c_{k,1}  = \min\left\{ n : \EE  [Z_{k+1, 1}] \ge \epsilon^{-1} \right\},\\
 b_{k,2} = \min\left\{ n : \EE  [Y_{k+1}]\ge \epsilon \right\}\,, \qquad 
c_{k,2}  = \min\left\{ n : \EE  [Y_{k+1}] \ge \epsilon^{-1} \right\}.
\end{gather*}

Then $\alpha_3(G_{n,1/2}) \in I_n$ with high probability where $I_n$ is the interval \begin{equation*}
    I_n = \begin{cases}
        \{2k\} &\text{if } a_k \leq n < b_{k,1} \\
        \{2k, 2k + 1\} &\text{if }b_{k,1} \leq n < c_{k,1} \\
        \{2k+1\} & \text{if }c_{k,1}\leq n < b_{k,2} \\
        \{2k+1,2k+2\} & \text{if }b_{k,2}\leq n < c_{k,2}  = a_{k+1}.
    \end{cases} 
\end{equation*}

For larger $r$, we recall that classical works of Erd\H{o}s-Kleitman-Rothschild \cite{EKR} and  Kolaitis-Pr\"omel-Rothschild~\cite{KPR} show that almost every $K_{r+1}$-free graph is $r$-partite.  
As $n$ varies, we will see that for most values of $n$ the subgraphs that realize $\alpha_{r+1}(G_{n,1/2})$ are in fact $r$-partite graphs.  As $n$ increases and we transition from having independence number $k$ to independence number $k+1$, there will be vertex sets of size $k+1$ that have a small number of edges.  As in the case of triangles, we refer to these edges as \emph{defects} in potential independent sets. It turns out that we can incorporate these $(k+1)$-sets with few defects into $K_{r+1}$-free subgraphs. We show
that this can be done by ``covering'' each individual defect with an independent set of size $k$, where a $k$-set $A$ \emph{covers} a
defect $e$ if there is no copy of $K_3$ that consists of $e$ and a vertex in $A$. The point of this definition is that any copy of $K_{r+1}$ that contains a defect from a $(k+1)$-set cannot have any vertex in a $k$-set that covers the defect; this allows us to include both these sets in a $K_{r+1}$-free subgraph. It turns out that for $n$ in this regime we can form a maximum $K_{r+1}$-free subgraph of $G_{n,1/2}$ by identifying disjoint sets $A_1, \dots, A_r$ where each set $A_i$ is either an independent set of size $k$ or a set of size $k+1$ with a small number of defects with the property that each defect is covered by an independent set of size $k$ in the collection.  It will also be the case that each $k$-set covers only one defect. We show that in this regime $X$ is equal to the maximum
number of vertices in such a structure.  We do not show that all $K_{r+1}$-free subgraphs of maximum size have this form; indeed, in some cases different---but closely related---maximum structures are also possible.

\subsection{Concentration via Poissonian statistics} \label{sec:Poissonian}
We now prepare to describe our main result.  As a typical $K_r$-free graph is $(r-1)$-partite it will be more notationally convenient to work with $K_{r+1}$-free subgraphs. For $r = o(\log \log n / \log\log\log n)$, set $X = \alpha_{r+1}(G_{n,1/2}).$ In analogy with the case of triangles, define  
\begin{gather*}Y_k =
\text{number of independent sets of size }k \\
 Z_{k,i} = \text{number of }k\text{-sets
of vertices that contain exactly }i\text{ edges}.   
\end{gather*} 
For each $ j \in \{1, \dots, r\}$ define
\[ \mu_j = \left\lfloor \frac{r-j}{j} \right\rfloor = \left\lfloor \frac{r}{j} \right\rfloor -1 \ \ \ \text{ and } \ \ \
\xi_j = j ( \mu_j +2) -r. \]
Note that we have $ 1 \le \xi_j \le j$.  Just as we were occasionally able to form a triangle-free graph by combining an independent set of size $k$ with a $k+1$-set with a single defect, we will see that there are values for $n$ 
in the regime where $ \E[ Y_{k+1}]$ goes to zero sufficiently slowly (to be
precise, $n$ in the regime defined by 
$ \E[Y_{k+1}] \approx ( \log n)^{- 2\mu_j} $) 
for which we can form a $K_{r+1}$-free subgraph in $ G_{n,1/2}$
with $kr + j$ vertices by using $\xi_j$ many $(k+1)$-sets that have $\mu_j$ defects each and $j - \xi_j$ many $(k+1)$-sets with $ \mu_j+1$ defects each. Each defect
in this structure is covered by an independent set of size $k$, and it
follows that the total number of $k$-sets and $(k+1)$-sets involved
in the structure is
\[ \xi_j ( \mu_j +1) + (j - \xi_j)( \mu_j +2) =  j( \mu_j+2) - \xi_j  = r. \] 
Note that we cannot form such a structure with fewer than $ \xi_j$ many $(k+1)$-sets with $\mu_j$ defects each as this would require too many $k$-sets to cover all of the defects.
Of course, in order for such a structure to 
appear $n$ must be large enough for $(k+1)$-sets
with $ \mu_j$ defects to appear in $ G_{n,1/2}$. Note that as $n$ increases the smallest number of defects in a $(k+1)$-set decreases.  
Our main result is that structures of the type discussed above govern the
evolution of $X$; the appearance of these structures
is in turn governed by the number of
$(k+1)$-sets with few defects; and the number of
$(k+1)$-sets with a particular number of defects has the Poissonian statistics that we would anticipate. 

\begin{theorem} \label{thm:poisson-version}
  Let $k = 2 \log_2 n + O(\log\log n)$, $r = o(\log \log n/\log\log\log n)$. Then, for $n$  sufficiently large and $j \in \{1, \dots, r\}$ we have 
    $$\left|\PP(X \geq kr + j)  - \PP(Z_{k+1,\mu_{j}} \geq \xi_{j})\right| \leq  (\log n)^{-1/2}\,.$$
    For $\lambda=\E [Z_{k+1,\mu_{j}}] \leq n^{1/4}$ we also have $$\left|\PP(\Poiss(\lambda) \geq \xi_{j})   - \PP(Z_{k+1,\mu_{j}} \geq \xi_{j})\right| \leq n^{-1/4}\,.$$
\end{theorem}
\noindent See Figure \ref{fig:K12} for a depiction of some of the structures that govern the evolution of $ \alpha_{r+1}(G_{n,1/2})$ in the case $r=11$.

\begin{figure}[!ht] 
    \centering
 \includegraphics[width=0.88\textwidth]{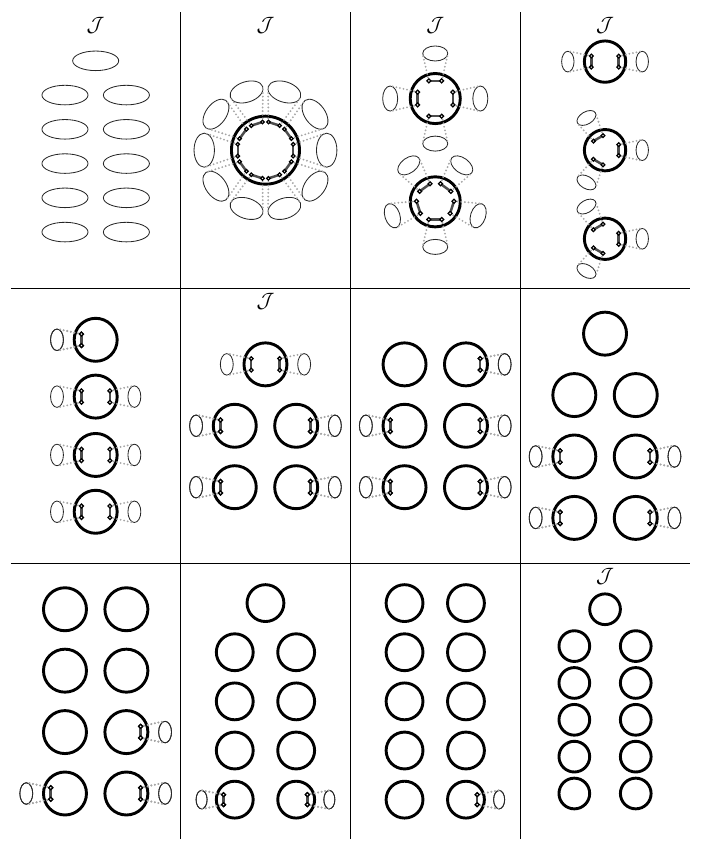}
\caption{ These images depict likely maximum induced subgraphs of $ G_{n,1/2}$ not containing $ K_{12}$ at different points in the evolution of $ G_{n,1/2}$ (where we view $n$ as growing and the edge probability is fixed at $1/2$). The black circles represent sets of $k+1$ vertices that have some number of defects, which are the gray edges. The gray ovals are independent sets of size $k$, and the dotted lines indicate which $k$-sets cover which defects. Note that 
we have $ \mu_1=10, \xi_1=1; \mu_2= 4, \xi_2=1; \mu_3=2, \xi_3=1; \mu_4=1, \xi_4=1; \mu_{5}=1, \xi_5=4$ and  $ {\mathcal J} = \{1,2,3,5,11\}$. We placed a $\mathcal{J}$ for sets whose size lies in $\mathcal{J} \mod 11$. These sizes form the endpoints of the intervals of concentration of $X$.}
\label{fig:K12}
\end{figure}

\subsection{The interval of concentration for \texorpdfstring{$K_{r+1}$}{Kr+1}} \label{sec:intervals}
We now deduce the extent of concentration of $ \alpha_{r+1}( G_{n,1/2})$ for fixed $r$
directly from \cref{thm:poisson-version}.  Implicit in \cref{thm:poisson-version} is that  for each $j \in \{1,2, \dots, r\}$ such that $ \mu_j \neq \mu_{j+1}$ there is an interval in $n$ in which $ \alpha_{r+1}(G_{n,1/2})$ is likely 
to be equal to $ rk + j$.
In order to make this precise, we 
define a sequence of values of $n$, much like we did for the $r = 3$ case in \cref{sec:triangles}.  
We begin by defining the point at
which it is likely that there is a $K_{r+1}$-free set of vertices in $G_{n,1/2}$ that consists of $r$ independent sets of size $k$. Note that $\EE[Y_k]={\binom{n}{k}}2^{-{\binom{k}{2}}}$.
Set\footnote{Note the choice of this particular function is made for
convenience. Any $ \epsilon = o_k(1)$ suffices.} $ \epsilon = \epsilon(k) = 1/\log k $  
and define
\begin{align*}
a_k &=\min \left\{ n : \EE [Y_k] \ge \epsilon^{-1} \right\}.
 \end{align*}
First, let us observe that 
the sequence $a_k$ is exponential in $k$. We begin by noting that the bound $ \binom{n}{k} \le ( ne/k)^k$ implies $ a_k > (k/e)2^{(k-1)/2}$. Next observe that if $n > 2^{(k-1)/2}$,  then 
\[ \E[Y_k] = \left( \frac{ne}{k} 2^{ - \frac{k-1}{2}}\right)^k e^{O(k^2/2^{k-1}) + O( \log k )}.\] 
This implies
\[
a_k = (1 +o_k(1)) \frac{k}{e} 2^{\frac{k-1}{2}} \ \ \
\text{ and } \ \ \  
a_{k+1} = ( \sqrt{2} + o_k(1)) a_k.\]
We will see that for the majority of the values for $n$ in the interval $[a_k, a_{k+1})$ we have $X=rk$ with probability
at least $1 - O(\epsilon) $. Then, as $n$ approaches $a_{k+1}$, the structure of a maximum $K_{r+1}$-free subgraph goes through a series of transitions that interpolate between a subgraph of a Tur\'an graph on $kr$ vertices and a subgraph of 
a Tur\'an graph on $ (k+1)r$ vertices. In order to capture this behavior, 
we need some additional definitions. For $ a_k \le n < a_{k+1}$ define
\[ \mu = \mu(n) = \min \left\{i : \EE [Z_{k+1,i} ] > \epsilon^{-1} \right\},\]

and note that $\mu(n+1) \le \mu(n)$. 
Recall that we define
\[ \mu_j = \left\lfloor \frac{r-j}{j} \right\rfloor = \left\lfloor \frac{r}{j} \right\rfloor-1, \]
and by \cref{thm:poisson-version} we expect to have $K_{r+1}$-free subgraphs with $kr+j$ vertices with some significant probability for those $j$ such that $ \mu \le \mu_j$. Note that the quantity $ \mu(n)$ decreases from $ \Theta( \log n/ \log \log n)$  to $0$ as $n$ increases from $a_k$ to $a_{k+1}$ (see Fact~\ref{fact:numerical} that appears later in the paper for the detailed calculations that imply this). Thus, for the vast majority of the interval there is no improvement relative to the bound $ \alpha_{r+1}( G_{n,1/2}) \ge kr$, since sets of size $k+1$ will have too many defects to be included in our $K_{r+1}$-free subgraph. Then, as $n$ approaches $ a_{k+1}$, we see values of $\mu(n)$ that equal $\mu_j$ for $ 1 \le j\le r$, and it is here that we see increases in $ \alpha_{r+1}( G_{n,1/2})$. There is a series of changes in this value as $ \mu(n) $ decreases which we now describe.

Let $ \cJ $
be the set of integers $ j \in \{1, \dots ,r\} $ such that $ \mu_j \neq \mu_{j+1}$  (where we set $ \mu_{r+1} = -1$) and set $ s = | \cJ|$. For example, if $r=11$, then $\cJ=\{1,2,3,5,11\}$
and if $r=23$, then $\cJ=\{1,2,3,4,5,7,11,23\}$.
Write
$$ \cJ = \{j_i: i \in [s] \} = \{ 1=j_1<j_2<\cdots < j_s = r\},$$
and for each $i \in [s]$ set
\begin{align*}
b_{k,i}  = \min\left\{ n : \EE  [Z_{k+1, \mu_{j_i}}]\ge \epsilon \right\}\,,\qquad 
c_{k,i} = \min\left\{ n : \EE  [Z_{k+1, \mu_{j_i}}] \ge \epsilon^{-1} \right\},
\end{align*}
and note that $c_{k,s} = a_{k+1}$.
Observe that these definitions imply the following chain of inequalities.  The calculations needed to justify these inequalities may be found in \cref{fact:numerical}. 
$$a_k \le b_{k,1} \le c_{k,1} \le b_{k,2} \le c_{k,2} \le \cdots \le b_{k,s-1}\le c_{k, s-1} \le b_{k,s} \le c_{k,s} = a_{k+1}.$$
We are now ready to define the intervals on which $X$ is concentrated. Set $ j_0 = 0$ and  $ c_0 = a_k$.
Define
\[ I_n = \begin{cases}   
    \left\{ kr+ j_{i-1}  \right\}    & \text{if } c_{k,i-1} \le n < b_{k, i} \text{ and } 1 \le i \le s  \\
   \left\{ kr + j_{i-1}, \dots, kr+ j_{i} \right\}  & \text{if } b_{k,i} \le n < c_{k,i}  \text{ and } 1 \le i \le s 
\end{cases}\]
 In other words, if $n \in [b_{k, j_i}, b_{k, j_{i+1}})$, then $I_n$ is either
$\left\{ kr + j_{i-1}, \dots, kr+ j_i \right\}$ or  $\left\{ kr+ j_i  \right\}$ according to whether $n < c_{k,j_i}$ or $n \ge c_{k,j_i}$.  

{\bf Example.} If $r=11$, then $\cJ=\{1,2,3,5,11\}$,  $s=5$,  and $j_1=1, j_2=2, j_3=3, j_4=5, j_5=11$. Consequently, 
\[ I_n = \begin{cases}   
    \{ 11k\} & \text{if } a_k = c_{k,0} \le n < b_{k,1} \\
 \{ 11k, 11k+1\} & \text{if } b_{k,1} \le n < c_{k,1} \\
 \{11k+1\} & \text{if } c_{k,1} \le n < b_{k,2} \\
\{ 11k+1, 11k+2\} & \text{if } b_{k,2} \le n < c_{k,2} \\

\{11k+2\} & \text{if } c_{k,2} \le n < b_{k,3} \\

\{ 11k+2, 11k+3\} & \text{if } b_{k,3} \le n < c_{k,3} \\

\{11k+3\} & \text{if } c_{k,3} \le n < b_{k,4} \\

\{ 11k+3, 11k+4, 11k+5\} & \text{if } b_{k,4} \le n < c_{k,4} \\

\{11k+5\} & \text{if } c_{k,4} \le n < b_{k,5} \\

\{11k+5, \ldots, 11(k+1)\} & \text{if } b_{k,5} \le n < c_{k,5}= a_{k+1}. \\
    \end{cases}\]
See Figure~\ref{fig:K12} for a depiction of some of the 
structures involved in each step here.

With this preparation, we now state our result regarding the concentration of $ \alpha_{r+1}( G_{n,1/2})$.

\begin{theorem} \label{thm:clique}
Fix $r \ge 1$. Then whp $\alpha_{r+1}(G_{n, 1/2}) \in I_n$.
\end{theorem}

\begin{proof}
\cref{thm:clique} will follow immediately from \cref{thm:poisson-version}. 
Indeed, setting $j = j_{i-1} + 1 $ we have $ \mu_j =\mu_{j_{i}}$ and 
\begin{multline*} n < b_{k,i} \ \ \implies \ \ \E[ Z_{k+1, \mu_{j_{i}}}] < \epsilon
\ \ \implies \ \ \E[ Z_{k+1, \mu_{j}}] < \epsilon \\ \ \ \implies \ \ 
\PP( X >rk + j_{i-1})= \PP(X \ge rk+j) \le \PP( \Poiss( \epsilon) \ge 1) +o_n(1) = 1 - e^{-\epsilon} +o_n(1) = O( \epsilon).
\end{multline*}
And, for $k$ sufficiently large we have
\begin{multline*} n \ge c_{k,i} \ \ \implies \ \ \E[ Z_{k+1, \mu_{j_{i}}}] > 1/\epsilon \\
\ \ \implies \ \ 
\PP( X \ge rk+j_i) \ge \PP( \Poiss( 1/\epsilon) \ge r) -o_n(1) = 1 - \frac{ (1/\epsilon)^{r}}{ r!}e^{-1/\epsilon} - o_n(1) = 1 - o_k(1).
\end{multline*}
Consequently, the implications above imply that if $c_{k, i-1} \le n < b_{k,i}$, then
$$\PP(X=kr+j_{i-1}) = \PP(X\ge kr+j_{i-1}) -\PP(X>kr+j_{i-1}) \ge 
1-o_k(1)-O(\varepsilon)$$ and if $b_{k, i} \le n < c_{k,i}$, then, since $c_{k, i-1}\le b_{k, i} \le n < c_{k,i} \le b_{k, i+1}$, we have
\begin{equation*}\PP(X>kr+j_{i}) = O(\varepsilon) \qquad \hbox{and}\qquad  \PP(X<kr+j_{i-1}) = o_k(1). \qedhere \end{equation*}
\end{proof}

The proof of \cref{thm:poisson-version} is split into upper and lower bounds.  Showing a lower bound on $X \geq kr + j$ will be fairly simple: we show that one can take $\xi_j$ many $(k+1)$-sets with $\mu_{j}$ defects, $j-\xi_j$ many $(k+1)$-sets
with $ \mu_j+1$ defects, and $r - j$ many $k$ independent sets and combine them to get a $K_{r+1}$-free set on $kr + j$ vertices.  The upper bound is more delicate. 
One important tool in the proof of the upper bound is a result of Balogh  and Samotij~\cite{balogh2019efficient} which counts the number of $K_{r+1}$-free graphs on $k$ vertices that are far from being $r$-partite. Their argument is an application of an efficient hypergraph container lemma (see~\cref{lem:far}) and can be viewed as a generalization and refinement of the classical works of Erd\H{o}s-Kleitman-Rothschild~\cite{EKR} and Kolaitis-Pr\"omel-Rothschild~\cite{KPR} who showed that almost all $K_{r+1}$-free graphs are $r$-partite.

\subsection{Color-critical graphs and other properties} \label{sec:other-properties}

Rather than looking only at the largest $K_r$-free subgraph, we can similarly examine the largest $F$-free graph, where we use $F$-free to mean that there is no (not necessarily induced) subgraph  isomorphic to $F$.  With this in mind, define $\alpha_F(G)$ to be the size of the largest $F$-free subgraph of $G$.

Say that a graph is $r$-{\em color-critical} (henceforth $r$-{\em critical}) if $\chi(F)=r+1$ and $F$  contains an edge whose deletion reduces its chromatic number to $r$. For example, $C_5$ is 2-critical. While we are not able to deduce as refined a statement for $\alpha_F$ for color-critical graphs, we are able to show concentration on an interval of  size at most $r{+1}$.

\begin{theorem}\label{thm:color-critical-concentration}
    Fix $r \geq 2$ and an $r$-critical graph $F$.  There is a sequence $m_n$ so that with high probability $\alpha_F(G_{n,1/2}) \in [m_n-r,m_n]$.
\end{theorem}

\begin{remark}
    The proof of \cref{thm:color-critical-concentration} yields a somewhat explicit choice of $m_n$.  Let $m_0$ be the smallest integer so that the expected number of $F$-free graphs in $G_{n,1/2}$ whose vertex set has size $m_0$ is at most $1$.  Then we take either $m_n = m_0 -1$ or $m_n = m_0 - 2$ depending on the expected number of $F$-free subgraphs of size $m_0$.
\end{remark}

As an example, Theorem~\ref{thm:color-critical-concentration} implies that whp $\alpha_{C_5}(G_{n,1/2}) \in \{m_n-2, m_n-1, m_n\}$.  The proof of \cref{thm:color-critical-concentration} is much simpler: the upper bound is given by using a theorem of Pr\"omel and Steger enumerating such graphs along with a basic union bound; for the lower bound, the theorem of Pr\"omel and Steger will essentially say that the number of such graphs is the same as the number of $r$-partite graphs, and so we will find an $r$-partite graph using the same ideas as in the proof of \cref{thm:poisson-version}. 

More generally, a \emph{property} $\mathcal{P}$ of graphs is a class of graphs closed under isomorphism. Given a property $\mathcal{P}$ and a graph $G$, let $\alpha_{\mathcal{P}}(G)$ be the maximum number of vertices in an induced subgraph of $G$ that is isomorphic to some member of $\mathcal{P}$.
A graph property $\mathcal{P}$ is \emph{monotone} (respectively \emph{hereditary}) if every (induced) subgraph of a graph with property $\mathcal{P}$ also has property $\mathcal{P}$. For instance, being a bipartite graph or being a triangle-free graph is monotone, while being perfect or chordal is hereditary. To every monotone (hereditary) property $\mathcal{P}$, there exists a collection $\mathcal{F}$ of graphs such that $G \in \mathcal{P}$ if and only if every (induced) subgraph of $G$ is not isomorphic to any member of $\mathcal{F}$. 
We end with  the following problem, which asks if for every monotone or hereditary property there is concentration on a constant-length interval.

\begin{problem} \label{conj:main}
    Let $\mathcal{P}$ be a monotone or hereditary property. Does there exist a positive integer $R=R_{\mathcal{P}}$ such that with high probability $\alpha_{\mathcal{P}}(G_{n, 1/2})$ lies in an interval of length at most $R$?
\end{problem}

We believe this is true for all properties $\mathcal{P}$ for which the number of graphs on $n$ vertices in $\mathcal{P}$ is $2^{\Omega(n^2)}.$  

\section{Preparations for proof of Theorem \texorpdfstring{\ref{thm:poisson-version}}{1}}

We now begin discussing various ingredients needed for our proof of Theorem~\ref{thm:poisson-version}.  Throughout the rest of the paper we will assume that $n$ is sufficiently large.  We will first compute various relevant scales and then import a few useful tools such as Janson's inequality and a consequence of the Stein-Chen method on Poisson approximation.

\subsection{Useful asymptotics}

Recall that $Z_{k,i}$ is the number of $k$-sets with exactly $i$ edges and $Y_k = Z_{k,0}$ is the number of independent sets of size $k$.  We start with a basic calculation for understanding the scales at play. 

\begin{fact}\label{fact:numerical}
    We have $\E [Y_k] = n^{O(1)}$ if and only if $k = 2\log_2 n - 2 \log_2\log_2 n + O(1).$  If $k = 2 \log_2 n + O(\log \log n)$ then we also have $\E [Y_{k+1}]/\E [Y_k] = n^{-1 + o(1)}$.  If we further have $i = O(\log \log n)$ then $$\E [Z_{k,i}] = k^{2i(1 + o(1))} \E [Y_k]$$
\end{fact}
\begin{proof}
    Write $k = 2\log_2 n - 2 \log_2 \log_2 n - t$ and compute \begin{align*}
        \E [Y_k] = \binom{n}{k} 2^{-\binom{k}{2}} = \exp_2\left((1 + o(1))k[\log_2 n - \log_2 k - k/2 + O(1)]\right) = n^{(1+o(1))t + O(1)}\,.
    \end{align*}
    The second assertion is similar.  For the third assertion, note that $$\E[ Z_{k,i}] = \binom{n}{k} \binom{\binom{k}{2}}{i}2^{-\binom{k}{2}} = k^{2i(1 + o(1))} \E [Y_k] $$
    completing the proof.
\end{proof}

We will also make frequent use of a second moment-type calculation:

\begin{fact}\label{fact:Y-k-Delta}
    If $k = 2\log_2 n + O(\log \log n)$ then $$\binom{n}{k}\sum_{j = 1}^{k-1} \binom{k}{j}\binom{n-k}{k-j}2^{-2\binom{k}{2}+ \binom{j}{2}} \leq n^{-1 + o(1)}( \E [Y_k]^2 + \E [Y_k]) \,. $$
\end{fact}
\begin{proof}
   For $j \in [k-1]$, write $C_j= {\binom{k}{j}}{\binom{n-k}{k-j}} 2^{\binom{j}{2}}$ so that \begin{align*}
        \binom{n}{k}\sum_{j = 1}^{k-1} \binom{k}{j}\binom{n-k}{k-j}2^{-2\binom{k}{2}+ \binom{j}{2}} = \binom{n}{k} 2^{-2\binom{k}{2}} \sum_{j = 1}^{k-1} C_j.
    \end{align*}
    Bound  \begin{equation*}
        \frac{C_{j+1}}{C_j} = \frac{k-j}{j+1} \frac{\binom{n-k}{k-j-1}}{\binom{n-k}{k-j}} 2^j \leq \log^3 n \cdot n^{-1} 2^{j}\,.
    \end{equation*}
   Above we added an extra factor of $\log n$ to absorb all implicit constants.  For $2 \leq j \leq (3/4)k$ we then have 
   \begin{align*}
        \frac{C_j}{C_1} = \prod_{i = 1}^{j-1} \frac{C_{i+1}}{C_i} \lesssim \left(\frac{\log^3 n \cdot 2^{j/2}}{n}\right)^{j-1} \leq n^{-1 + o(1)}\,.
    \end{align*}
    Similarly bound \begin{equation*}
        \frac{C_{k-(j+1)}}{C_{k - j}} \leq \log^3 n \cdot n \cdot 2^{-(k-j)}.
    \end{equation*}
    For $2 \leq k - j \leq (3/4)k$ we then have $$\frac{C_{k - j}}{C_{k-1}} \leq n^{-1 + o(1)}\,.$$
    This implies \begin{align*} \binom{n}{k}\sum_{j = 1}^{k-1} \binom{k}{j}\binom{n-k}{k-j}2^{-2\binom{k}{2}+ \binom{j}{2}} &\leq (1 + o(1))\binom{n}{k} 2^{-2\binom{k}{2}}(C_1 + C_{k-1}) \\
    &= (1 + o(1))\binom{n}{k} 2^{-2\binom{k}{2}}\left(k{n-k \choose k-1} + k(n-k)2^{{k-1 \choose 2}}\right) \\ 
    &= n^{-1 + o(1)}( \E [Y_k]^2 + \E [Y_k])\,. \qedhere
    \end{align*}
\end{proof}

\subsection{Poisson Approximation}

Define $\Poiss(\lambda)$ to be the Poisson distribution of mean $\lambda$ and set $d_{\TV}(X,\Poiss(\lambda))$ to be the total variation distance to the Poisson of mean $\lambda$, i.e.\ $$d_{\TV}(X,\Poiss(\lambda)) = \sum_{k \geq 0}\left|\PP(X = k) - \frac{e^{-\lambda} \lambda^k}{k!} \right|.$$

The Stein-Chen method will allow us to show that the number of copies of certain subgraphs will be approximately Poisson.  We use the following simplified version from \cite[Thm.\ 1]{arratia1990poisson}

\begin{theorem} \label{thm:stein-chen}
   For $\alpha \in \cA$, let $X_\alpha$ be  a Bernoulli random variable and set  $p_\alpha = \E [X_\alpha] = \PP(X_\alpha = 1)$.
     Let $W = \sum_{\alpha \in \cA} X_\alpha$. For each $\alpha$, let $B_\alpha = \{\beta : X_\alpha \text{ and } X_\beta \text{ are not independent}\}.$   Then 
     $$d_{\TV}(W, \Poiss(\E [W]))  \leq 2\left(\sum_{\alpha \in \cA} \sum_{\beta \in B_\alpha}p_\alpha p_\beta + \sum_{\alpha \in \cA} \sum_{\beta \in B_\alpha \setminus \alpha} \E[X_\alpha X_\beta] \right)\,.$$
\end{theorem}

Our main use of \cref{thm:stein-chen} is to show Poissonian statistics for $Z_{k,i}$.

\begin{lemma}\label{lem:Poisson-approx-Z}
    If $k = 2\log_2 n + O(\log \log n)$ and $i = O\left(\log n \log n\right)$ then $$d_{\TV}(Z_{k,i},\Poiss(\E [Z_{k,i}])) \leq n^{-1 + o(1)} ( \E [Z_{k,i}]^2 + \E [Z_{k,i}])\,. $$
\end{lemma}
\begin{proof}
    Let $\lambda = \E[Z_{k,i}] = \E[Y_k] n^{o(1)}$.  Set $\cA = \binom{[n]}{k}$ to be the collection of $k$ sets of $[n]$ and for each $\alpha \in \cA$ set $X_\alpha$ to be the indicator that $\alpha$ has exactly $i$ edges; note that $Z_{k,i} = \sum_{\alpha \in \cA} X_\alpha$.  Further, note that $p_\alpha$ is constant over all $\alpha$ and we have $p_\alpha  = \lambda / \binom{n}{k} = \binom{\binom{k}{2}}{i} 2^{-\binom{k}{2}}\,.$   Note that we have    $\binom{\binom{k}{2}}{i} = n^{o(1)}$.
    For each $\alpha$, we have that $B_\alpha$ is the collection of sets that overlap with $\alpha$ in at least $2$ vertices.  We can then bound \begin{align*}
        \sum_{\alpha \in \cA} \sum_{\beta \in B_\alpha \setminus \alpha} \E[X_\alpha X_\beta] \leq  n^{o(1)} \binom{n}{k} \sum_{j = 2}^{k-1} \binom{k}{j} \binom{n-k}{k-j} 2^{-2\binom{k}{2} + \binom{j}{2}} \leq n^{-1 + o(1)} ( \E [Z_{k,i}]^2 + \E [Z_{k,i}])
    \end{align*}
    by Fact \ref{fact:Y-k-Delta}.  A similar argument shows $$\sum_{\alpha \in \cA} \sum_{\beta \in B_\alpha}p_\alpha p_\beta  \leq n^{-1 + o(1)} ( \E [Z_{k,i}]^2 + \E [Z_{k,i}])$$ and so Theorem \ref{thm:stein-chen} completes the proof.
\end{proof}

\subsection{Janson's inequality}
We make use of not only the classical Janson's inequality but also a
generalization due to Riordan and Warnke \cite{riordan2015janson}. The set-up for our application
of these inequalities is as follows.

Let $S$ be a finite set, consider the probability space on
the $ \Omega = \{0,1\}^S$ given by the product measure, and let $ \cI$ be
the collection of all decreasing events in $ \Omega$. This implies that for all $A \cap B \in \mathcal{I}$, we have $\PP(A \cap B) \geq \PP(A)\PP(B)$, $A\cup B \in \mathcal{I}$, and $A\cup B \in \mathcal{I}$.  Given 
$A_1, \dots, A_k \in \cI$ we let $ I_i$ be the indicator random variable
for the event $A_i$ and set
\[ X = \sum_{i=1}^k I_i \ \ \ \text{ and } \ \ 
\Delta = \sum_{i}\sum_{j \sim i} \PP( A_i \cap A_j)\]
where we write $ i \sim j$ if $ i\neq j$ and $A_i$ and $ A_j$ are dependent.

\begin{theorem}[Janson; Riordan and Warnke] \label{thm:janson}
Under the conditions above, for any $0 \le t \le \mu$ we have
\[ \PP(X \le \EE[X] - t ) \le e^{- \frac{ t^2}{  2(\mu+  \Delta)}}.\]
\end{theorem}
\noindent Note that we recover Janson's inequality by specifying a collection of sets $ B_1, B_2 \dots, B_k \subset \Omega$ and letting the event $A_i$ be the event no element of $B_i$ appears.

\section{The lower bound: covering defects with independent sets} \label{sec:LB}

Our goal is to prove  the following that gives the lower bound for \cref{thm:poisson-version}. Recall that $\mu_j = \lfloor (r-j)/j \rfloor$ and $ \xi_j = (\mu_j +2)j - r$. 
\begin{lemma}\label{lem:new-lower}
    Let $n$ be sufficiently large and suppose $k = 2\log_2 n - 2 \log_2 \log n + O(1), r = o(\log\log n)$.  
    If $ 0 < j \le r$ 
    then $$\PP(X \geq kr + j) \geq \PP(Z_{k+1,\mu_{j}} \geq \xi_j) - \frac{1}{\log n} \,.$$
\end{lemma}

We begin by noting that we may assume
\begin{equation}
\label{eq:newcondition}
\frac{1}{ 2 \log n} < \E[Z_{k+1, \mu_j}] < \log^{4r} n
\end{equation}
Indeed, if $ \EE[Z_{k+1, \mu_j}] \leq (2\log n)^{-1}$ then, by \cref{lem:Poisson-approx-Z} we have 
\[ \PP( Z_{k+1, \mu_j} \ge \xi_j) \le \PP(Z_{k+1, \mu_j} \ge 1) \le  n^{-1+o(1)} + \left(1 - e^{-\E[ Z_{k+1, \mu_j}]} \right)  \le \frac{1}{\log n}, \]
and the conclusion of \cref{lem:new-lower} is trivial. To see
the upper bound, note that we have
\begin{equation}
\label{eq:Zratio}
\frac{ \E[ Z_{k+1, i+1}] }{ \E[ Z_{k+1, i}]} = \frac{ \binom{k}{2}- i }{ i+1} \le \log^4 n . 
\end{equation}
It follows that if $ \EE[Z_{k+1, \mu_j}] \ge \log^{4r} n$
then $\E[ Z_{k+1,0}] = \E[Y_{k+1}] > \log^{4} n$. We may assume 
$ \E[ Y_{k+1}] \le n^{1/4}$ by simply restricting our attention to a subset of the vertex set. Then \cref{lem:Poisson-approx-Z} implies that it is likely that there are at $r$ independent sets of size $k+1$, and a first moment calculation using \cref{fact:Y-k-Delta} shows that these independent sets are disjoint whp. So
we have a $ K_{r+1}$-free set by taking the union of $r$ pairwise-disjoint independent sets of
size $k+1$; to be precise, we have
\[ \PP( X \ge kr+j) \ge \PP( X \ge (k+1)r) \ge  1 - n^{-1/2+ o(1)}. \]
As $ \E[ Z_{k+1,i}] \ge \log^4 n =: \lambda$ implies
\[  \PP( Z_{k+1,i} \ge r) \ge 1 - \sum_{\ell = 0}^{r-1} e^{-\lambda} \frac{ \lambda^\ell}{ \ell!} \ge 1 - e^{ -( \log n)^2 + O( (\log \log n)^2)},  \]
this bound suffices to establish the conclusion of \cref{lem:new-lower} in this
case. 

Note that assumption (\ref{eq:newcondition}) implies
\[ \E[ Y_k] = n^{1 + o(1)}. \]

Now, 
for each $k$ and $i$, define $\cZ_{k,i}$ to be the collection of $k$-sets of $[n]$ that induced exactly  $i$ edges so that $Z_{k,i} = |\cZ_{k,i}|$.  Similarly define $\cY_k = \cZ_{k,0}$ to be the collection of independent sets of size $k$. 
We will show that elements of $\cZ_{k,i} \cup \cZ_{k,i+1}$ are pairwise disjoint in the relevant regime. Define the ``bad'' event $\cB_1$ via $$\cB_1 = \{\exists~S_1,S_2 \in \cZ_{k+1,i} \cup \cZ_{k+1,i+1}, S_1 \neq S_2, S_1 \cap S_2 \neq \emptyset\}\,.$$

A simple union bound will handle $\PP(\cB_1)$:
\begin{fact}\label{fact:Z_k-disjoint}
    For $k = 2 \log_2 n + O(\log \log n)$ and $i = o(\log n / \log\log n)$ we have $$\PP(\cB_1) \leq n^{-1+o(1)}( \E [Z_{k+1,i}]^2 + \E [Z_{k+1,i}]).$$
\end{fact}
\begin{proof}
    As $i= o(\log n /\log\log n)$, we have $k^i=n^{o(1)}$. We union bound over all possible choices and use  \cref{fact:Y-k-Delta} to bound \begin{align*}
    \PP(\cB_1) &\leq n^{o(1)}  \binom{n}{k+1} \sum_{j = 1}^{k} \binom{k+1}{j} \binom{n-k-1}{k+1-j} 2^{-2\binom{k+1}{2} + \binom{j}{2}} \\
    &\leq n^{-1 + o(1)} ( \E [Z_{k+1,i}]^2 + \E [Z_{k+1,i}])\,. \qedhere
    \end{align*}
\end{proof}

For $S \in \cZ_{k,i}$ let $\cD(S)$ denote the set of defects in $S$.  Recall that an independent set $T \in \cY_k$ \emph{covers} an edge $e$ if there is no triangle consisting of a vertex in $T$ together with $e$.  
We prove that each defect is covered by many independent sets of size $k$. Set $\cC(e) = \{T \in \cY_k : T \text{ covers }e\}$.  Define the bad events $\cB_2$ and $\cB_3$ via \begin{gather*}
    \cB_2 := \left\{\exists~S \in \cZ_{k+1,i} \cup \cZ_{k+1,i+1}, e \in \cD(S) : |\cC(e)| \leq (3/4)^k \E [Y_k] / 2 \right\} \\
    \cB_3 := \left\{\exists~S \in \cZ_{k+1,i} \cup \cZ_{k+1, i+1}, e \in \cD(S), T_1, T_2 \in \cC(e) : T_1 \cap T_2 \neq \emptyset, T_1 \neq T_2\right\}\,.
\end{gather*}

\begin{lemma}\label{lem:large-covers}
    Suppose that $\E [Y_k] \in [n^{5/6},n^{6/5}]$.  Then for $i = O(\log\log n)$ we have 
    $$\PP(\cB_2 ) \leq  n^{-1/100}\quad \text{ and }\quad  \PP(\cB_3) \leq n^{-1/100}\,.$$
\end{lemma}
\begin{proof}
    Let $S \in \cZ_{k+1,i} \cup \cZ_{k+1, i+1}$ and $e \in \cD(S)$.  We begin with a union bound:  \begin{align} \label{unionb3}
        \PP(\cB_2) \leq (i \E[Z_{k+1,i}] + (i+1)\E[ Z_{k+1, i+1}]) \PP_{n-(k+1)}\left( |\cC(e)| \leq (3/4)^k \E[Y'_k] / 2 \right)
    \end{align}
   where we write $\PP_{n-(k+1)}$ for the distribution of the random graph on $ \binom{V}{2} \setminus \binom{S}{2}, e \in S$ and
   $ Y_k'$ is the number of independent sets of size $k$ that do not intersect $S$.

    To handle this probability, we will show by the second moment method that the number of independent sets that cover the disjoint edge $e$ is near its mean.  Define $N = n- (k +1)$.  Set $X_S = | \{ T \in \cC(e) : S \cap T = \emptyset\}|$ and note \begin{equation} \label{eq:E_N[X]}
        \E[ X_S] = (3/4)^k \binom{N}{k} 2^{-\binom{k}{2}} = (1 + o(1)) (3/4)^k \E [Y_k']\,.
    \end{equation}  We also can bound \begin{align*}
        \mathrm{Var}[X_S] - \E[X_S] &\le \binom{N}{k} \sum_{j = 1}^{k-1} \binom{k}{j} \binom{N - k}{k - j}2^{-2\binom{k}{2} + \binom{j}{2}} (3/4)^{2k-j} \\
        &= N^{-1 + o(1)} \left(\E[X_S]^2 + \E[X_S]\right)
    \end{align*}
    where the second line is by adapting \cref{fact:Y-k-Delta} and we recall $\E[X_S] = (1 + o(1))(3/4)^k \E[Y_k']$ by \eqref{eq:E_N[X]}.  By Chebyshev's inequality, we then see \begin{align*}
        \PP(X_S \leq (3/4)^k \E [Y_k'] / 2) \leq (4 + o(1))\frac{\mathrm{Var}[X_S]}{\E [X_S]^2} \leq  4 \left(N^{-1 + o(1)} + \frac{1 + o(1)}{\E[X_S]} \right) \leq n^{-1/99}
    \end{align*}
    where in the last line we used $5/6 + 2\log_2(3/4) \geq 1/99$. Consequently, 
   $$ \PP_{n-(k+1)}\left( |\cC(e)| \leq (3/4)^k \E[Y'_k] / 2 \right)=\PP(X_S \leq (3/4)^k \E [Y_k'] / 2) \leq n^{-1/99}.$$
   Finally, since $\EE[Z_{k+1, i}], \E[ Z_{k+1,i+1}]$, and $i$ are each $n^{o(1)}$, we obtain 
$\PP(\cB_2 ) \leq  n^{-1/100}$ from (\ref{unionb3}).
   
    For the bound on $\cB_3$ we may again union bound to see \begin{align*}
        \PP(\cB_3) &\leq (i \E[Z_{k+1,i}] + (i+1) \E[ Z_{k+1, i+1}]) \binom{N}{k} \sum_{j = 1}^{k-1} \binom{k}{j} \binom{N - k}{k - j}2^{-2\binom{k}{2} + \binom{j}{2}} (3/4)^{2k-j} \\
        &= n^{-1 + o(1)}( i\E[Z_{k+1,i}] + (i+1) \E[ Z_{k+1, i+1}]) \left(\E[X_S]^2 + \E[X_S]\right)\,. 
    \end{align*}
    As $ \E[ Y_k'] = (1+o(1)) \E[Y_k]$, recalling (\ref{eq:E_N[X]}) we have
    $$\E[X_S] = n^{2\log_2(3/4) + o(1)} \E[Y_k] \le n^{2\log_2(3/4) + 6/5+o(1)}
    \le n^{0.38 +o(1)}.
    $$ As $\E[Z_{k+1,i}], \E[Z_{k+1,i+1}]$ are each $n^{-1 + o(1)} \E[Y_k]$ this completes the proof.
\end{proof}

As our structure may include some $(k+1)$-sets with $ \mu_j+1$ defects, we include a fourth bad event $ \cB_4$, which we define to
be the event
$ Z_{k+1, \mu_j+1} < j - \xi_j$. Note that assumption (\ref{eq:newcondition}) and observation (\ref{eq:Zratio}) imply $E[ Z_{k+1, \mu_j +1}] \geq (\log n)/(2r)$. We then apply \cref{lem:Poisson-approx-Z} to conclude 
\begin{equation} \label{eqn:larger sets}
\PP( \cB_4) < n^{-1 +o(1)} + \PP( {\rm Poiss}[(\log n)/(2r)] \le r)  < n^{-1 + o(1)} + e^{ -( \log n)/(2r)} r (\log n)^r < e^{-( \log n)^{1/2}}.
\end{equation}
\begin{proof}[Proof of \cref{lem:new-lower}]  
    Set $\cB = \cB_1 \cup \dots \cup \cB_4$ and note that by \cref{fact:Z_k-disjoint},  
    \cref{lem:large-covers} and \eqref{eqn:larger sets} we have that $\PP(\cB) \leq e^{ -( \log n)^{1/2} +o(1)}$.  
    Define the ``good'' event  $\cG$ to be the event $ Z_{k+1,\mu_{j}} \geq \xi_j$. Then \cref{lem:Poisson-approx-Z} and \cref{fact:numerical} show \begin{equation} \label{eq:lower-bound-Poisson-approx}
    \PP(\cG) \geq \PP(\Poiss(\E [Z_{k+1,\mu_{j_i}}]) \geq {\xi_j}) - n^{-1/3}\,.
    \end{equation} We now note that on the event $\cG \cap \cB^c$, we have at least $\xi_j$  elements of $\cZ_{k+1,\mu_{j}}$
    and at least $j - \xi_j$ elements of $ \cZ_{k+1, \mu_j+1}$ and these sets are pairwise disjoint by event $\cB_1^c$. Let $ A_1, \dots, A_j$ be this collection of $(k+1)$-sets. Each defect in these sets is covered by at least $(3/4)^k \E[Y_k]/2 = n^{1 + 2 \log_2(3/4) + o(1)} \geq n^{3/20}$ elements of $\cY_k$ by $\cB_2^c$; 
    by event $\cB_3^c$, each pair of independent sets that both cover a defect among $\cZ_{k+1,\mu_{j}} \cup \cZ_{k+1, \mu_j+1}$ are disjoint.  We may thus construct a $K_{r+1}$-free graph on $kr+j$ vertices by combining the sets $ A_1, \dots, A_j$ with
    \[ \xi_j\mu_j + (j - \xi_j)(\mu_j+1) = j( \mu_j+1) - \xi_j = r-j \]
    covering independent $k$-sets.  
    
    To see that there is no $K_{r+1}$ among these vertices, suppose $K$ is a set of $r+1$ vertices in this structure and assume for the sake of contradiction that $K$ spans a copy of $K_{r+1}$. For $i=1, \dots, j$ let $|K \cap A_j| = k_j$. As the graph induced by $K$ is complete this implies that $ K \cap A_j$ contains $ \binom{k_j}{2}$ defects. The $k$-sets in our structure that cover these defects are not in $K$. As each $k$-set
    intersects $K$ in at most one vertex, it follows that we 
    have
    \[ r+1 = |K| \le \sum_{i=1}^j k_i + r-j - \sum_{i=1}^j \binom{k_i}{2}
    = r + \sum_{i=1}^j k_i - 1 - \binom{k_i}{2} \le r.\]

We conclude that \begin{equation*}
    \PP(X \geq kr + j) \geq \PP(\cG \cap \cB^c) \geq \PP(\cG) - \PP(\cB) \,. \end{equation*}
    Applying \eqref{eq:lower-bound-Poisson-approx} completes the proof.
\end{proof}

\section{The upper bound: not covering too many defects}

The main goal of this section is to prove the following probabilistic upper bound.

\begin{lemma}\label{lem:new-upper}
    Let $n$ be sufficiently large, $k = 2 \log_2 n + O( \log \log n)$ and $ r = o(\log\log n / \log\log\log n)$.     Then $$\PP(X \geq kr + j) \leq \PP(Z_{k+1,\mu_{j}} \geq \xi_{j}) + (\log n)^{-1/2}\,.$$
\end{lemma}

\begin{remark}
    We note that the error here is only polylogarithmic rather than polynomial.  While we make no effort to optimize the error, the error in the upper bound of \cref{lem:new-upper} in some instances must be only polylogarithmic.  This is because  the proof of \cref{lem:new-lower} can also be used to show $$\PP(X \geq kr + j) \geq \PP(Z_{k+1,\mu_{j}} \geq \xi_j) - \PP(Z_{k+1,\mu_{j} +1} < j - \xi_j ) - n^{-\Omega(1)}\,.$$
    Depending on the value of $n,j_i$ and $j_{i-1}$, it is possible for both of the probabilities on the right-hand-side to be polylogarithmic.  The error in \cref{lem:new-upper} is driven by this event and comes into play with \cref{lem:no-fewer-defects}.
 
\end{remark}

Note that the desired conclusion of \cref{lem:new-upper} is equivalent, by taking complements,  to showing $$\PP(Z_{k+1,\mu_{j_i}} < \xi_{j_i}) \leq \PP(X < kr + j_i) + (\log n)^{-1/2}\,.$$

If $\E[Z_{k+1,\mu_{j}}] =\lambda \geq \log n$ note first that we may assume that $\lambda \leq n^{o(1)}$ by restricting our attention to a subset of the vertex set so that this occurs. Then by \cref{lem:Poisson-approx-Z} we have
$$\PP(Z_{k+1,\mu_{j}} < \xi_j)\le\sum_{t=0}^{\xi_j-1}\frac{e^{-\lambda} \lambda^t}{t!}+d_{\TV}(Z_{k+1, \mu_{j}}, \Poiss(\lambda)) \leq n^{-1 + o(1)},$$
where the last inequality holds as $\lambda < n^{o(1)}$. This shows the conclusion of \cref{lem:new-upper}. Consequently, to prove \cref{lem:new-upper} we may assume that $\E[Z_{k+1,\mu_{j}}] \leq \log n$.

To begin with, we show that 
each $(k+1)$-set has at least $\mu_{j} $ defects.  Set 
\begin{equation} \label{eq:fewer-than-mu-defects}
    \cB_1' = \bigvee_{\mu < \mu_{j}}\{ Z_{k+1,\mu} > 0\}\,.
\end{equation}

\begin{lemma}\label{lem:no-fewer-defects}
     Let $k = 2\log_2 n + O(\log\log n), r = o(\log\log n / \log\log\log n)$ and $\mu_{j}$ satisfy $\E[ Z_{k+1,\mu_{j}}] \leq \log n$. Then $$\PP(\cB_1') \leq (\log n)^{-1 + o(1)}\,.$$
\end{lemma}
\begin{proof}
    By \cref{fact:numerical}  we may bound \begin{equation*} 
        \PP(\cB_1') \leq \sum_{\mu < \mu_{j}} \E [Z_{k+1,\mu}] \leq \sum_{\mu < \mu_{j}} \E [Z_{k+1,\mu_{j}}] \cdot k^{-2(\mu_{j} - \mu) + o(1)} \leq  (\log n)^{1-2+o(1)}<(\log n)^{-1 + o(1)}\,. \qedhere
    \end{equation*}
\end{proof}

We prove \cref{lem:new-upper} by a stability argument. 
The first step in the proof of \cref{lem:new-upper} is an application of an argument of Balogh
and Samotij~\cite{balogh2019efficient} that bounds the number of $ K_{r+1}$-free graphs on a given set of vertices. As we need to 
adapt their argument for our purpose, we begin with one of their definitions. We say
that a graph $G$ is {\em $t$-close to $r$-partite} if $G$ can be made $r$-partite by removing
$t$ edges. We say that $G$ is {\em $t$-far from $r$-partite} if there is no such set of $t$ edges. 

Define $\cB_2'$ to be the event that for some $kr \leq u \leq (k+1)r$ there is a $K_{r+1}$-free subgraph on $u$ vertices that is $k^{3/4}$-far from being $r$-partite.  

\begin{lemma} \label{lem:far}
If $k = 2\log_2 n + O(\log\log n)$ and $r = o(\log\log n / \log\log\log n)$ then
we have $\PP( \cB_2') \le  e^{ -\omega( \log n)}$.   
\end{lemma}
\cref{lem:far} is essentially implicit in the work \cite{balogh2019efficient}.  More specifically,~\cite{balogh2019efficient} implicitly proves the following counting statement.   Write $\mathrm{ex}(m,K_{r+1})$ for the maximum number of edges in a $K_{r+1}$-free graph on $m$ vertices. 
\begin{proposition}\label{prop:EKR-BS}
    Let $\cG$ be the collection of $K_{r+1}$-free graphs on $m$ vertices that are $m^{3/5}$-far from being $r$-partite with $r = o(\log m / \log\log m)$.  Then \begin{equation}
        | \cG| \le 2^{{\rm ex}(m, K_{r+1})}  e^{ - m (\log m)^3/2}.
    \end{equation} 
   \end{proposition} 

We deduce \cref{prop:EKR-BS} from \cite{balogh2019efficient} in \cref{sec:far}, where we make no attempt to optimize the logarithms.  From here, a simple union bound handles \cref{lem:far}.

\begin{proof}[Proof of \cref{lem:far}]
    We will handle a fixed $u \in [kr,(k+1)r]$ first and let $\cB_2'(u)$ denote the event that there is a $K_{r+1}$-free subgraph on $u$ vertices that is $k^{3/4}$-far from being $r$-partite. Write $ u =kr +a$.  Let us argue that
    \begin{equation} \label{eqn:akru}
    \frac{ \binom{n}{u}}{ \binom{n}{k}^{r-a} \binom{n}{k+1}^{a}} \EE[ Y_k]^{r-a} \EE[ Y_{k+1}]^a
        <\EE[Y_k]^r.
    \end{equation}
    First, assume $a=0$. In this case we are to show that ${n \choose kr}< {n \choose k}^r$.
    As $1 \ll k = O(\log n)$, we certainly have ${n \choose k}>(2n/k)^k$, and since $r \ge 2$, this yields
    ${n \choose kr}\le (ne/kr)^{kr}<(2n/k)^{kr} < {n \choose k}^r$. Next, assume $a>0$. In this case, apply \cref{fact:numerical} to obtain
     \begin{equation*} 
    \frac{ \binom{n}{u}}{ \binom{n}{k}^{r-a} \binom{n}{k+1}^{a}} \EE[ Y_k]^{r-a} \EE[ Y_{k+1}]^a
        <\frac{ \binom{n}{u}}{ \binom{n}{k}^{r-a} \binom{n}{k+1}^{a}} \EE[Y_k]^{r} n^{-a+o(a)}.
         \end{equation*}
         The power of $n$ in the expression on the right above (excluding the $\EE[Y_k]$ term) is
         $$u-a+o(a)-k(r-a)-(k+1)a=kr-k(r-a)-(k+1)a+o(a) = -a+o(a)<0,$$
         and hence (\ref{eqn:akru}) holds.
         
    Next, observe that ${u \choose 2} = {\rm ex}(u,K_{r+1}) +(r-a) \binom{k}{2} + a \binom{k+1}{2}$. Let $\cF$ denote the set of $K_{r+1}$-free graphs on $u$ vertices that are $k^{3/4}$-far from being $r$-partite. 
     Then (\ref{eqn:akru}) yields
      \begin{align*}
    \PP( \cB_2'(u)) &\le \binom{n}{u} | \cF| \, 2^{- \binom{u}{2}} \\
    &= |\cF| \binom{n}{u}
    2^{- {\rm ex}(u,K_{r+1}) -(r-a) \binom{k}{2} - a \binom{k+1}{2}} \\
    &= | \cF| 2^{-{\rm ex}(u, K_{r+1})} \frac{ \binom{n}{u}}{ \binom{n}{k}^{r-a} \binom{n}{k+1}^{a}} \EE[ Y_k]^{r-a} \EE[ Y_{k+1}]^a \\
    &\le \frac{ | \cF|}{ 2^{{\rm ex}(u, K_{r+1})}} \EE[Y_k]^r. 
       \end{align*}
    Every graph in $\cF$ is certainly $u^{3/5}$-far from being $r$-partite as $u^{3/5} < ((k+1)r)^{3/5} < k^{3/4}$. Hence by Proposition~\ref{prop:EKR-BS}, $|\cF|\le 2^{{\rm ex}(u, K_{r+1})}  e^{ - u (\log u)^3}$. As $k \sim 2 \log_2n$ and $r = o(\log\log n)$, we finally obtain 
    \begin{equation*}   \PP( \cB_2'(u)) \le e^{ - u (\log u)^3} \EE[Y_k]^r
    \leq \exp( -kr (\log kr)^3 + 2r \log n) 
    \leq \exp(-\omega(\log n))\,. \qedhere
    \end{equation*}
\end{proof}

In light of Lemma~\ref{lem:far} we restrict our attention to potential $K_{r+1}$-free graphs
with vertex sets $ A_1 \dot\cup A_2 \dot\cup \dotsm \dot\cup A_r$ where each set $A_i$ contains
at most $ k^{3/4}$ edges.  We say a set is \emph{light} if it has at most $k^{3/4}$ edges.

A simple first moment computation shows that with high probability there are no light sets on $m$ vertices where $m \ge k+2$.  Let $\cB_3'$ be the event that there is a light $(k+2)$-set.

\begin{fact}\label{fact:no-light-k+2}
    If $\E [Y_k] \leq n^{1 + o(1)}$ then $\PP(\cB_3') \leq n^{-1 + o(1)}$.  
\end{fact}
\begin{proof}
    By \cref{fact:numerical},  \[ \PP(\cB_3') =\binom{n}{k+2} \binom{ \binom{k+2}{2}}{ k^{3/4}} \left( \frac12 \right)^{ \binom{k+2}{2}} 
= \EE[ Y_{k+2}] e^{ O( (\log k) k^{3/4}) } = \EE[ Y_{k+2}] n^{o(1)} \leq n^{-1 + o(1)} \,. \qedhere\]
\end{proof}

Since we assume that $|A_1 \cup \cdots \cup A_r| \ge kr$ and $\max |A_i| \le  k+1$ we  can restrict our attention to light sets $A_i$ with $ k -(r-1) \le |A_i| \le k+1$.  Further, recall that we refer to each edge in such a light set as
a \emph{defect}.

We now make observations that put additional conditions on the defects in the
sets $ A_1, \dots, A_r$. We start by noting that a collection of
$r-1$ pairwise-disjoint light sets spans many copies of $ K_{r-1}$. Let $\cB_4'$ be the
event that there are pairwise disjoint light sets $ A_1, A_2, \dots, A_{r-1}$ 
and sets $ B_i \subset A_i$ such that $ k - (r-2) \le |A_i| \le k+1$ and $B_i \ge k^{2/3}$ such that there is no copy of 
$K_{r-1}$ with exactly one vertex in each $B_i$
\begin{lemma} \label{lem:multipartite}
  If $k = 2\log_2 n -2\log_2\log_2 n + O(1)$ and $ r = o(\log\log n / \log\log\log n)$ then we have $ \PP( \cB_4') = o(1/n)$.
\end{lemma}
\begin{proof}
Note that the claim holds trivially for $r=2$. For larger $r$
the proof is a straightforward application of Janson's inequality. We will union bound over choices of $A_{1},\ldots,A_{r-1}$ and $B_1,\ldots,B_{r-1}$.

Let $ r \ge 3$ and let $ B_1, \dots, B_{r-1}$ be fixed, disjoint sets
of $ k^{2/3}$ vertices. Let the random variable $W$ count the number of
$ K_{r-1}$'s among these sets with one vertex in each $ B_{i}$. We have 
$\EE[W] = k^{2(r-1)/3} 2^{-{\binom{r-1}{2}}} $. 
Furthermore, letting $\Delta$ denote the corresponding quantity for an application of Janson's inequality, we may bound
$$ \Delta \le \EE[W] \sum_{ j=2}^{r-2} \binom{r-1}{j} k^{\frac{2(r-j-1)}{3}} \left( \frac{1}{2} \right)^{\binom{r-1}{2} - \binom{j}{2}}
= \EE[W]^2 \sum_{j=2}^{r-2} \binom{r-1}{j} 2^{\binom{j}{2}} k^{-\frac{2j}{3}} 
= r^{O(1)} \EE[W]^2 k^{-4/3}.$$
Janson's inequality (\cref{thm:janson}) then implies that the probability that the sets $ B_1, \dots, B_{r-1}$ span no
copy of $ K_{r-1}$ is at most $ e^{ - \Omega( k^{5/4})}$.

Next we note that the expected number of collections of pairwise disjoint light sets $ A_1, \dots, A_{r-1}$ that contain sets $ B_1, \dots, B_{r-1}$, respectively, such
that $ |B_i| = k^{2/3}$ is at most
\[  \left[ \sum_{\rho=-1}^{r-2} \EE[ Y_{k- \rho}] \binom{ \binom{k - \rho}{2}}{ k^{3/4}}\binom{k- \rho}{k^{2/3}}  \right]^{r-1} <
\left[ r \EE[ Y_{k-r+1}] \binom{ \binom{k+1}{2}}{ k^{3/4}}\binom{k+1}{k^{2/3}}  \right]^{r-1} = e^{ O(r^2k)}. \]
The desired bound on $ \PP( \cB_4')$ then follows from the first moment.
\end{proof}

\noindent
Observe that if we have a defect in one of the sets $ A_1, \dots ,A_r$ then we
expect that edge to form copies of $K_3$ with many (about $k/4$) of the vertices
in each of the other $A_i$. \cref{lem:multipartite} then implies that in such a situation
we find a copy of $K_{r+1}$ in $ A_1 \dot\cup A_2 \dot\cup \dotsm \dot\cup A_r$. So, if we find defects among the sets $A_1, \dots, A_r$ then we would
be in the situation (which we might naively anticipate is unlikely) that for each of the defects there is one of the sets not containing the defect with the property
that the defect forms few copies of $K_3$ with the vertices in that set. Now, it is important to note that the probability that a defect makes
\emph{no copies} of $K_3$ with one of the sets is not negligible. Indeed, we take
advantage of this fact in the proof of \cref{lem:new-lower}.

Given a collection of pairwise disjoint light sets $ A_1, \dots, A_r$ we say that a set $A_j$ \emph{weakly covers} a defect in a set
$A_i$ if the defect forms a $K_3$ with at most $ k^{2/3}$ of the vertices in $A_j$. 
In light of \cref{lem:multipartite}, we may assume
that if $ A_1 \dot\cup \dotsm \dot\cup A_r$ defines a $K_{r+1}$-free subgraph of $ G_{n,1/2}$ then every defect in 
$ A_1, \dots, A_r$ is weakly covered by one of the sets in the collection.

We first observe that only sets of size smaller than $k+1$ can weakly cover a defect in a $(k+1)$-set.  Let $\cB_5'$ be the event that there exist light $(k+1)$-sets $A$ and $B$ such that $A$ weakly covers a defect in $B$.  

\begin{fact}\label{fact:k+1-do-not-cover}
     If $\E [Y_k] \leq n^{1 + o(1)}$ then $\PP(\cB_5') = e^{-\Omega(k)}$
\end{fact}
\begin{proof}
    The probability that a given light $(k+1)$-set $A$ weakly covers a defect $e$ in a light $(k+1)$-set $B$ is at most $$\binom{k+1}{k^{2/3}} \left( \frac{3}{4} \right)^{k + 1 - k^{2/3}} = \left(\frac{3}{4} \right)^{k(1 + o(1))}\,.$$
    We now union bound over all potential configurations using $\E[Y_{k+1}]=n^{-1+o(1)}\E [Y_k] = n^{o(1)}$:
    \begin{equation*}
        \PP(\cB_5') \leq \EE[ Y_{k+1}]^2 \binom{ \binom{k+1}{2}}{ k^{3/4}}^2 k^{3/4} \left(\frac{3}{4} \right)^{k(1+o(1))} = \left(\frac{3}{4}\right)^{k(1 + o(1))}\,. \qedhere
    \end{equation*}
\end{proof}

As our goal here is to bound $X$ relative to the simple lower bound $kr$, our focus is on defects in sets $A_i$ such that $ |A_i| = k+1$. Our canonical picture is that each of
these defects is covered by an independent set $ A_\ell$ such that $ |A_\ell| =k$,
where each such $k$-set covers at most one defect. But we emphasize that other structures are possible that occur with roughly the same probability. 

For example, consider 
$ j \in \cJ$ such that $  \xi_j \le j -2$ (e.g. the case $r=11, j=3, \mu_3=2, \xi_3=1$ depicted in Figure~\ref{fig:K12} and shown again in Figure \ref{fig:two-options}).

\begin{figure}[ht]
    \centering
  \includegraphics[width=0.9\textwidth]{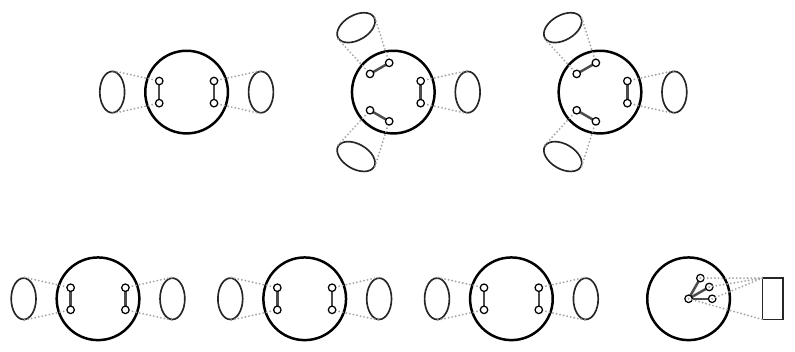}
\caption{Two structures that could achieve $ X \ge rk+3$ in the case $r=11$ (where we seek a $K_{12}$-free set) are shown. Recall that $\mu_3=2, \xi_3=1$, and whp 
$ Z_{k+1, 2} \ge 1$ implies that the standard structure, which is the
structure in the first row, appears whp. However, if $ Z_{k+1,2} \ge 3$ we could also achieve $ X \ge kr+3$ with the second structure
depicted here. In this second structure the gray rectangle is an independent
set of size $k-1$ that covers the three defects in the fourth $(k+1)$-set.
}
\label{fig:two-options}
\end{figure}

Our canonical picture, shown on the top line of Figure \ref{fig:two-options}, consists of $\xi_j = 1$ many $(k+1)$-sets with $\mu_j = 2$ defects each together with  $j - \xi_j = 2$ many $(k+1)$-sets with $\mu_{j} + 1 = 3$ defects.  Each of the defects is covered by its own $k$-set.  However, there are other ways to obtain a $K_{r+1}$-free set of size $kr + j$.  One of them is to combine three $(k+1)$-sets with $\mu_j = 2$ defects, each of which is covered by a $k$-set, together with one $(k+1)$-set with $3$ defects all of which are covered by a single $(k-1)$-set.  A typical $k$-set is not likely to cover more than one defect, however there are sufficiently many $(k-1)$-sets that it is possible for one of these to cover more than one defect.  Such an adjustment does not influence $X$ but leads to flexibility in the maximum structure. This trade-off---of introducing some $(k+1)$-sets with more than $\mu_j$ defects at the expense of using sets smaller than $k$ to cover them---is captured in Lemma \ref{lem:light-sets-covering-defects},  which is the technical core of the proof of \cref{lem:new-upper}.

For the remainder of the proof we
fix a set of at most $ \mu<r$ defects in each of the $ (k+1)$-sets among the $A_i$. 
We prove that each $(k-\rho)$-set $ A_\ell$  weakly covers at most $ 1 + (\mu+1) \rho$ of the specified 
defects in the $ (k+1)$-sets. To
be precise, let $\cE = \cE_{\mu, \rho, \beta}$ be the event that there is
a light set $ A$ such that $ |A| = k - \rho$ where $ 0 \le \rho \le r-2$ 
and edges $e_1, e_2, \dots, e_\beta$ not contained in $A$ such that
\begin{enumerate}
\item $e_i$ is an edge in a light set $A_i$ such that $ |A_i| =k+1$ and $A \cap A_i = \emptyset $ for $i=1,\dots, \beta$
\item if $i \neq \ell$ then $ A_i $ and $ A_\ell$ are either equal or disjoint, 
\item \label{it:at-most-mu} each set $A_i$ contains at most $ \mu$ of the edges $ e_1, \dots, e_\beta$,
\item at most $ k^{2/3}$ of the vertices in $A$ form a triangle with $e_i$ for $i =1, \dots , \beta$,  and
\item \label{it:beta-LB} $\beta > 1 + (\mu+1) \rho$.
\end{enumerate}

\begin{lemma}\label{lem:light-sets-covering-defects}
 If $\E [Y_k] \leq n^{1 + o(1)}$ then $ \PP[ \cE] \le n^{-1/8 + o(1)}$.
\end{lemma}
\begin{proof}
We bound the expected number of such configurations. 
Consider a fixed $ \rho $ and the connected components in the graph formed by the $ \beta$ defects. Let $ \beta_1$ be the number
of isolated edges in this graph and let $ \beta_2$ be the number of connected components with at least 2 edges.
We claim that \cref{it:at-most-mu}~and~\cref{it:beta-LB} imply that we have
\begin{equation} \frac{ \beta_1}{2} + \beta_2 \ge \rho + 1. \label{eq:betas} \end{equation}
Indeed, if $ \mu =1 $ then $ \beta_2 = 0$ and (\ref{eq:betas}) follows immediately from \cref{it:beta-LB}. If $ \mu \ge 2$ then observe
that we have
\[ \beta_1 + \mu \beta_2 \ge \beta \ge 2 + (\mu + 1) \rho \ \ \ \Rightarrow \ \ \
\frac{ \beta_1}{2} + \beta_2 \ge \rho+ \frac{ \rho +2 - \beta_1 }{ \mu} + \frac{ \beta_1}{2}
\]
Noting that the left hand side of (\ref{eq:betas}) is in $ {\mathbb Z}^+ \cup \{0,1/2\}$, we see that (\ref{eq:betas}) follows by considering the cases $ \beta_1=0$; $\beta_1=1$; $ 2 \le \mu \le 4$; and finally $ \mu \ge 5$ and $ \beta_1 \ge 2$.

Recall that the probability that a fixed edge makes no copy of $K_3$ with a set of $t$ vertices is $ (3/4)^t$. Similarly, the probability that no edge in some fixed connected graph with at least two edges makes no copy of $K_3$ with the vertices in a set of $t$ vertices is at most $ (5/8)^t$.
Thus, for a given light set $A$ on $k - \rho$ vertices and a given collection of $(k+1)$-sets with defect graphs having $\beta_1$ isolated edges and $\beta_2$  connected components with at least $2$ edges, the probability that at most $k^{2/3}$ of the vertices in $A$ form a triangle with the defect graph is bounded above by 
\begin{align*}
    \binom{k-\rho}{k^{2/3}}^\beta \left( \frac{3}{4}  \right)^{\beta_1(k - k^{2/3})} \left( \frac{5}{8} \right)^{\beta_2(k - k^{2/3})} &\leq \left(\frac{9}{16} \right)^{(\beta_1 + o(1))\log_2 n}\left(\frac{25}{64} \right)^{(\beta_2 + o(1))\log_2 n} \\
    &\leq \left(\frac{2}{5} \right)^{(\frac{\beta_1}{2} + \beta_2 + o(1)) \log_2 n}\,.
\end{align*}
Write $Z_{k,\le i}$ for the number of $k$-sets with at most $i$ edges. Then by \cref{fact:numerical} 
$$ \E[Z_{k - \rho, \leq k^{3/4}}] \leq n^{o(1)} \E[Y_{k - \rho}] = n^{1 + \rho + o(1)}\qquad \text{ and } \qquad \E[ Z_{k+1, \leq k^{3/4}}] = n^{o(1)}\,.$$ The number of choices for the set $A$ together with the tuple of at most $\beta_1 + \beta_2$ light sets of size $k+1$ together with choices for the configuration of defect graph is bounded above by \begin{align*}
    \E[Z_{k - \rho, \leq k^{3/4}}] \sum_{j = 1}^{\beta_1 + \beta_2}\left[ \E[Z_{k+1,\leq k^{3/4}}]\binom{k^{3/4}}{\mu} \right]^j = n^{1 + \rho + o(1)} \cdot \sum_{j = 1}^{\beta_1 + \beta_2}\left[ 2^{O( (\log k)k^{3/4})} \right]^j = n^{1 + \rho + o(1)}
\end{align*}
where we used that $\mu,\rho < r$.
Combining the two displayed equations provides the bound \begin{equation*}
    \PP[\cE] \leq \left(\frac{2}{5} \right)^{(\frac{\beta_1}{2} + \beta_2 + o(1)) \log_2 n} \cdot n^{1 + \rho + o(1)} \leq \left(n^{\log_2 (2/5) + 1} \right)^{(\frac{\beta_1}{2} + \beta_2 + o(1))} \leq n^{-1/8 + o(1)}
\end{equation*}
where in the first inequality we used \eqref{eq:betas} and the second we used that $\frac{1}{2}(\log_2 (2/5) + 1) < -1/8.$
\end{proof}

\begin{remark}
We note that some simple counting illustrates that the bound on $ \beta$ in 
the definition of $ \cE$ should be sufficient for our purposes. We would like to conclude that 
we cannot gain an advantage in the random variable
$X$ by replacing a $k$-set in the canonical structure with a smaller set. If we replace a $k$-set
with a $(k - \rho)$-set then we should also replace at least $\rho$ other $k$-sets with $(k+1)$-sets in order to achieve a balance relative to $X$. Thus, the
new set of size $ k - \rho$ would need to weakly cover the defects in $ \rho$ new $ (k+1)$-sets (of which
there are at least $ \rho \mu$). Furthermore, the new $ (k-\rho)$-set would need to cover the defects that were potentially weakly covered by the $ \rho +1$ sets of size $k$ that were previously in the canonical structure (one that was replaced with a smaller set and $ \rho$ that were replaced with
larger sets). Thus, in order to be useful, this $ (k -\rho)$-set should weakly cover at least $ \mu \rho + \rho + 2 $ defects. If the event $ \cE$ does not hold, then this type of exchange is not possible. 
\end{remark}

\begin{remark}
    The bound in \cref{it:beta-LB} is not tight in general, it
    is just what is needed for the proof of \cref{lem:new-upper}. For example, 
    if $\mu\ge 2$ then we could replace \cref{it:beta-LB} with $\beta > 1 + \mu \rho$ and we see that using smaller sets generally loses ground.  However, we cannot replace \cref{it:beta-LB} with $\beta > 1 + \mu \rho$ when $\mu= 1$ as this stops working for some moderate $\rho$.
\end{remark}

As a last quasirandomness event, we define $$\cB_6' = \bigcup_{\mu \geq \mu_{j}, \rho \in [0,r - 2],\beta \in [0,r^2]} \cE_{\mu,\rho,\beta}$$
where we are union bounding over all choices of $\mu,\rho$ and $\beta$ that (naively) are possible.  An immediate corollary to \cref{lem:light-sets-covering-defects} by using the union bound is the following:
\begin{corollary}\label{cor:count-covering}
    If $\E [Y_k] \leq n^{1 + o(1)}$ then we have $\PP(\cB_6') \leq n^{-1/8 + o(1)}$.
\end{corollary}
\begin{proof}
    There are at most $O(r^3)$ choices for $\beta,\rho$.  By \cref{it:beta-LB}, for each choice of $\beta \in [0,r^2]$ and $\rho \neq 0$ we have at most $O(r^2)$ choices of $\mu$.  Taking a union bound over these choices and applying \cref{lem:light-sets-covering-defects} completes the proof.
\end{proof}

We now have all the ingredients needed to quickly complete the proof of Lemma~\ref{lem:new-upper}.

\begin{proof}[Proof of \cref{lem:new-upper}]
    We assume that none of the events $\cB_1',\ldots,\cB_6'$ hold by \cref{lem:no-fewer-defects}, \cref{lem:far}, \cref{fact:no-light-k+2}, \cref{lem:multipartite}, \cref{fact:k+1-do-not-cover}, and \cref{cor:count-covering}.  We also assume that $Z_{k+1,\mu_{j}} \leq \xi_{j} - 1$ and want to show that $X \leq kr + (j-1)$.  By event $\cB_2'^c$, if we have $X \geq rk$ then we must have that it can be partitioned into $r$ light sets.  
    By $\cB_3'^c$, each set in the partition has size at most $k+1$.  Suppose we have a partition of pairwise disjoint light sets $A_1,\ldots,A_r$ so that $k - (r - 1) \leq |A_i| \leq k+1$ for all $i$ and so that $A_1 \cup A_2 \cup \ldots \cup A_r$ is $K_{r+1}$ free.  Let $b$ be the number of $A_j$ with size at most $k$ and reorder the sets so that $A_1,\ldots,A_b$ are these sets; for $i = 1,\ldots, b$ define $\rho_i=k-|A_i|$  so that $|A_i| = k - \rho_i$.  Let $a$ be the number of $A_i$ of size $k+1$.  The size of the set $A_1 \cup A_2 \cup \ldots \cup A_r$ is $$rk + a - \sum_{i = 1}^b \rho_i =: rk + \delta\,.$$
    Our goal is to show that $\delta \leq j - 1$.  As such, we may assume that $a \geq j$ since $\rho_i \geq0$ for all $i$.  
    Write $\mu = \mu_{j}$ for simplicity and note that  no $(k+1)$-set has fewer than $\mu$ defects by $\cB_1'^c$. 
    Since $Z_{k+1,\mu} \leq \xi_j - 1$, and $a$  sets among the $A_i$ have size $k+1$, at least  $a - (\xi_j - 1)$ sets among the  $A_i$ have at least $\mu+1$ defects.  
    Consequently, the total number of defects is at least 
    $$(\xi_{j} - 1)\mu + (a - (\xi_{j} - 1))(\mu + 1)\,.$$
    Applying event $\cB_6'^c$, each $A_i$ covers at most $1 + (\mu+1)\rho_i$ defects.  Since all defects are covered by $\{A_i\}_{i = 1}^b$ we must have 
    $$(\xi_{j} - 1)\mu + (a - (\xi_{j} - 1))(\mu + 1) \leq \sum_{i = 1}^b(1 + (\mu + 1)\rho_i)\,.$$
    Rearranging and recalling $a - \sum_i \rho_i = \delta$ gives 
    $$\delta \leq \frac{b + \xi_{j} - 1}{\mu + 1} = \frac{b + j( \mu+2)-r-1}{ \mu+1} = j + \frac{b + j - r -1}{\mu+1}\,.$$
    To complete the proof recall that we assume $ a \ge j$ which implies $ b + j \le r$. Therefore,
    $\delta < j$, as desired.
\end{proof}

\section{Proof of Theorem~\ref{thm:color-critical-concentration}}
In this section, we give the proof of Theorem~\ref{thm:color-critical-concentration}.
When we say that $G$ is $F$-free, we mean that $G$ contains no subgraph isomorphic to $F$. We do not impose any requirement on the subgraph being induced.

We prove Theorem~\ref{thm:color-critical-concentration} via a first moment calculation.  We need the following result of Pr\"omel and Steger~\cite{PS2} on the number of $F$-free graphs.

\begin{theorem} [Pr\"omel-Steger~\cite{PS2}] \label{thm:PSFfree}
Fix $r \ge 2$ and an $r$-critical graph $F$. Then for $m$ sufficiently large, the number of $F$-free graphs on vertex set $[m]$ is at most
$$2^{\left(1-\frac1r\right)\frac{m^2}{2} + m\log_2r +\Theta(\log_2m)}.$$
\end{theorem}

With this in mind, define $$N_m := 2^{\left(1-\frac1r\right)\frac{m^2}{2} + m\log_2r}\,,\qquad \overline{N}_{n,m} = \binom{n}{m} N_m 2^{-\binom{m}{2}}\,.$$

We will compare $\overline{N}_{n,m}$---and thus the expected number of $F$-free sets---to the expected number of $k$ independent sets.  Recall that $Y_k$ is the number of independent sets of size $k$ in $G_{n,1/2}.$

\begin{lemma}\label{lem:compare-to-IS}
    Suppose that $m = r k$ is so that $\overline{N}_{n,m} = n^{\Theta(1)}$.  Then $\overline{N}_{n,m} = n^{o(1)}\cdot (\E [Y_k])^r\,.$
\end{lemma}
\begin{proof}
First write \begin{align*}
    \overline{N}_{n,m} \cdot (\E [Y_k])^{-r} &=\binom{n}{m} \exp_2\left( \left(1 - \frac{1}{r} \right)\frac{m^2}{2} + m \log_2 r - \binom{m}{2}\right) \cdot \binom{n}{k}^{-r}\exp_2\left(r \binom{k}{2} \right)\,.
\end{align*}
 Since  $\overline{N}_{n,m} = n^{\Theta(1)}$ we have $m = O(\log n)$ and so we have \begin{align*}
    \binom{n}{m} \cdot \binom{n}{k}^{-r} = n^{o(1)} \left(\frac{e n}{m} \right)^m \left( \frac{e n }{k} \right)^{-kr} = n^{o(1)} r^{-m}
\end{align*}
where in the last equality we recalled that $m = kr$.  Noting that \begin{align*}
    \left(1 - \frac{1}{r} \right)\frac{m^2}{2}  - \binom{m}{2} + r \binom{k}{2} = 0
\end{align*}
completes the proof.
\end{proof}

Using \cref{thm:PSFfree} we will be able to find an expression for the value of $m_0$ in \cref{thm:color-critical-concentration} up to lower order terms.

\begin{lemma}\label{lem:m_0-calculation}
    Let $X_m$ be the number of $F$-free sets of size $m$.  Define $m_0 = \min\{ m : \E [X_m] \leq 1\}\,.$ Then $m_0 = 2r \log_2 n + O(\log \log n).$
\end{lemma}
\begin{proof}
    Set $m = \alpha \log_2 n$ for $\alpha = O(1)$ to be chosen later and note that \begin{align*}
        \overline{N}_{n,m} &= \binom{n}{m} \exp_2\left(\left(1 - \frac{1}{r} \right)\frac{m^2}{2} + m \log_2 r - \binom{m}{2} \right) \\
        &= \exp_2 \left(m \left(\log_2 e + \log_2 n - \log_2 m + \left(1-\frac{1}{r} \right)\frac{m}{2} + \log_2 r - \frac{m}{2} + \frac{1}{2}  + o(1) \right) \right) \\
        &= \exp_2 \left(m \left( \log_2 n - \log_2 m  -\frac{m}{2r}   + O_r(1) \right) \right) \\
        &=\exp_2 \left(m \left( \left(1- \frac{\alpha}{2r} \right)\log_2 n - \log_2 m  + O_r(1) \right) \right)\,.
    \end{align*}
    Applying \cref{thm:PSFfree} completes the proof.
\end{proof}

Our last calculation before proving \cref{thm:color-critical-concentration} is to show that $\E [X_m]$ changes by a factor of $n^{1 + o(1)}$ as $m$ changes, provided $m$ is near $m_0$. 

\begin{lemma}\label{lem:m_0-shift}
    Let $X_m$ be the number of $F$-free sets of size $m$.  Define $m_0 = \min\{ m : \E [X_m] \leq 1\}$.  For $|\ell| = O(1)$ we have 
    $$\frac{\E [X_{m_0}]}{\E [X_{m_0 + \ell}]} = n^{\ell + o(1)}\,.$$  
\end{lemma}
\begin{proof}
    By \cref{lem:m_0-calculation} we have that $m_0 = 2r \log_2 n + o(\log n).$  For any $m = (2r + o(1)) \log_2 n$, use  \cref{thm:PSFfree} to compute
    \begin{align*}\frac{\E [X_{m}]}{\E [X_{m + 1}]}  &= n^{o(1)} \frac{\overline{N}_{n,m}}{\overline{N}_{n,m+1}} \\
    &= n^{o(1)} \binom{n}{m} \cdot \binom{n}{m+1}^{-1} 2^{-m} \\
    &\qquad \times \exp_2\left(\left(1 - \frac{1}{r} \right)\frac{m^2}{2} + m \log_2 r -  \left(1 - \frac{1}{r} \right)\frac{(m+1)^2}{2} - (m+1) \log_2 r\right) \\
    &= n^{o(1)} \cdot \frac{1}{n} \cdot  \exp_2\left( \frac{m}{r} \right) \\
    &= n^{1 + o(1)} \,. 
    \end{align*}
    Iterating this bound $\ell$ times completes the proof.
\end{proof}

\begin{proof}[Proof of \cref{thm:color-critical-concentration}]
    Let $X_m$ be the number of $F$-free sets of size $m$ and define 
    $$m_0 := \min\{ m : \E [X_m] \leq 1\}\,.$$ We proceed in two cases depending on if $\E[ X_{m_0} ] \leq n^{-1/{2r}}$ or $\E[ X_{m_0}] \geq n^{-1/{2r}}.$  If $\E[ X_{m_0} ] \leq n^{-1/(2r)}$ then set $M = m_0$, otherwise set $M = m_0 + 1.$  In either case, \cref{lem:m_0-shift} implies that $\E[X_M] \leq n^{-1/(2r)}$ and so $\PP(\alpha_F(G_{n,1/2}) \geq M) = o(1)$.  Now, let $k$ be the largest integer so that $kr < M - 1$.  Note that $kr \in [M- r-1, M - 2]$ so we can write $kr=M-\ell$ where $ \ell \in \{2,3,\ldots,r+1\}$.  By \cref{lem:m_0-shift} we have 
 \begin{equation}\label{eq:shift-X-kr}
 \E[X_{kr}] = \E[X_{M-\ell}] = \E[X_{M}] \,n^{\ell+o(1)} < n^{2 + r+o(1)}.
 \end{equation}

    We claim that \begin{equation}\label{eq:X-kr-LB}
        \E[X_{kr}] \geq n^{1/2} \,.
    \end{equation}
    To see this, assume first that $M = m_0$ and note that we must have $\E[X_{m_0}] \geq n^{-1 + o(1)}$ by \cref{lem:m_0-shift}.  Noting that $\E[X_{kr}] = \E[X_{m_0}]\, n^{\ell+o(1)}$ by \eqref{eq:shift-X-kr} shows \eqref{eq:X-kr-LB} in this case.  If  $M = m_0+1$, then we have $\E[X_{m_0}] \geq n^{-1/2r}$ implying that $\E[X_{kr}] \geq \E[X_{m_0}] n^{\ell-1 + o(1)}>
    \E[X_{m_0}] n^{1 + o(1)}$ thus showing \eqref{eq:X-kr-LB}.

    Note that by \cref{thm:PSFfree} and \cref{lem:compare-to-IS} we have $$\E[X_{rk}] = n^{o(1)} \overline{N}_{n,m} = n^{o(1)} (\E[Y_k])^r$$
    which implies that $\E[Y_k] \geq n^{1/(2r) + o(1)}.$  The proof of \cref{lem:new-lower} for $j = r$ applied to $(k-1)$ shows that with high probability there are $r$ many disjoint independent sets of size $(k-1) + 1 = k$.  Taking the induced subgraph with this vertex set yields a graph with  $kr$ vertices that has chromatic number at most $r$, thus implying it is $F$-free.  This shows that $\alpha_F(G_{n,1/2}) \geq kr= M -\ell \ge (M-1) -r$ with high probability.  Together with the upper bound, this shows $$\alpha_F(G_{n,1/2}) \in [(M-1) -r, M-1]$$ with high probability, completing the proof.
\end{proof}                                                                                                    

\section{Concluding Remarks}

\textbf{Extending beyond $p = 1/2$.} 
It is a natural to study 
$\alpha_{r+1}(G_{n,p})$ for $p \ne 1/2$. It appears that there are two relevant transitions for $p$.  Define \begin{equation*}
    (1 - p_-) + p_-(1 - p_-)^2 = (1 - p_-)^{1/2} \quad \text{ and } \quad p_+ = \frac{\sqrt{5} - 1}{2}\,.
\end{equation*}
The value $p_- \approx 0.323$ is when it becomes possible for an independent $k$-set to cover more than one defect; $p_+$ is when it becomes unlikely for an independent $k$-set to cover any defects.  We suspect the relevant regimes are broken apart in terms of $p_-$ and $p_+$

\begin{problem}
    Is it the case that for fixed $p \in (p_-,p_+)$ a version of \cref{thm:poisson-version} and \cref{thm:clique} hold?  For fixed $p > p_+$ is it the case that $\alpha_{r+1}(G_{n,p})$ is witnessed by a Tur\'an graph?  What happens for $p < p_-$?
\end{problem}

It is possible that new interesting phenomenon may become visible for smaller $p$ that is a function of $n$, and we leave this as an open question as well.

\textbf{Sharpness of \cref{thm:clique}.} We gather here two closing remarks regarding the distribution of $ X = \alpha_{r+1}( G_{n,1/2})$ for $r$ fixed.

\cref{thm:clique} is best possible in the sense that the result does not hold 
with intervals $I_n$ that are smaller. To see this, consider the
values of $n$ defined by
$$  n_{k,i} = \min \left\{n : \EE_n[Z_{k+1, \mu_{j_i}}] > 1 \right\} $$
for each $j_i \in \cJ$.
Note that we have $ b_{k,i} < n_{k,i} < c_{k,i}$.
Set $ \lambda = \EE_{n_k}[Z_{k+1, \mu_{j_i}}] $ and note that 
$\lambda = 1+o(1)$. Then by \cref{thm:poisson-version} we have
\begin{align*} \PP_{n_{k,i}}( X = kr + j_{i-1} ) &= \PP_{n_{k,i}}( Z_{k+1, \mu_{j_i}} < \xi_{j_{i-1}+1}) + O(1/\sqrt{ \log n}) = \sum_{j=0}^{\xi_{j_{i-1}+1} -1} e^{-\lambda} \frac{ \lambda^i}{i!} + O( 1/\sqrt{\log n})\\
&> \frac{1}{e} + o(1).\end{align*}
Similarly, for $ \ell \in \{ j_{i-1}+1, \dots, j_i - 1\}$ we have
\begin{align*} \PP_{n_{k,i}}(X= kr + \ell)&=\PP_{n_{k,i}}(X \ge kr+ \ell) - \PP_{n_{k,i}}(X \ge kr+  \ell+1)= \sum_{t=\xi_\ell}^{\xi_{\ell+1}-1} e^{-\lambda}\frac{\lambda^t}{t!} + O(1/\sqrt{\log n}) \\
&= \sum_{t=\xi_\ell}^{\xi_{\ell+1}-1} \frac{e^{-1}}{t!} + o(1).
\end{align*}
And finally we observe that
\begin{align*} \PP_{n_{k,i}}(X=kr + j_i) &= \PP_{n_{k,i}}( Z_{k+1, \mu_{j_i}} \ge\xi_{j_i} ) + O( 1/ \sqrt{ \log n})> 
e^{-\lambda}\frac{\lambda^{ \xi_{j_i}}}{\xi_{j_i}!} + O(1/\sqrt{\log n})\\
&= \frac{e^{-1}}{ \xi_{j_i}!} + o(1).
\end{align*}
Thus, $\alpha_{r+1}( G_{n_{k,i},1/2})$ takes all values in $\{ j_{i-1}, \dots, j_i\} $ with probabilities
bounded away from zero, and $ \alpha( G_{n,1/2})$ is not concentrated on any smaller interval.

As a second note, if $ \ell \notin \cJ$ then $ \PP( X \equiv \ell \mod r)$ is bounded away from 1 for
$n$ sufficiently large. Indeed, if $ j_{i-1} < \ell < j_i $ then we have 
\[ \PP(X=kr+\ell)=\PP(X \ge kr+ \ell) - \PP(X \ge kr+ \ell+1) = \sum_{t=\xi_\ell}^{\xi_{\ell+1}-1} e^{-\lambda}\frac{\lambda^t}{t!} + O(1/\sqrt{\log n})
\]
where $\lambda = \EE[ Z_{k+1, \mu_\ell}]$. As $\xi_\ell >0$ and $ \xi_{\ell +1} \le r$ 
this expression is bounded away from 1 for all $\lambda$.

\textbf{A hitting time result.}  The proof of \cref{thm:poisson-version} can be used to establish
a hitting time version of this result.  
Let $r$ be fixed and let $ 1 \le j \le r$. Define $T_1$ to the first
step $n$ in which we have $ X \ge kr+j$ where $ X = \alpha_{r+1}( G_{n,1/2})$, and
define $ T_2$ to be the first step $n$ at which we have $ Z_{k+1, \mu_j} \ge \xi_j$.
\begin{corollary}
With high probability we have $ T_1= T_2$.
\end{corollary}
\begin{proof}
Recall that we define $ b_{k,j} = \min \{n : \EE[ Z_{k+1, \mu_j}] \ge \epsilon \}$
and $ c_{k,j} = \min \{n : \EE[ Z_{k+1, \mu_j}] \ge 1/\epsilon \}$ where $ \epsilon = o_k(1)$. First observe that $ \PP( T_2 \not\in[ b_{k+1,j}, c_{k+1,j}]) = o(1)$, and we can thus restrict our attention to $n$ in this interval.

Next observe that the events $ \cB_1', \dots, \cB_6'$ 
are vertex-increasing in the
sense that if $ G $ is in one of these events and we add a vertex to $G$ to 
form a graph $ G'$ then $ G'$ is in the event. We establish above that with high probability 
none of these events holds when $n = c_{k+1,j}$. Setting $ \cB' = \vee_{i=1}^6 \cB_i'$,
it follows that
with high probability $ \cB'^c$ holds for all $ n \in [ b_{k+1,r}, c_{k+1,r}]$. 
Following the proof of \cref{lem:new-lower}, 
we conclude that on the event $ \{ n < T_2 \} \wedge \cB'^c$ we have $ X < kr+j$, which implies
$ n < T_1$. So, with high probability $ T_1 \ge T_2$.

We make a similar argument to establish $ T_2 \ge T_1$ whp. 
The events $ \cB_1$ and $ \cB_3$ are vertex-increasing while
the event $ \cB_4$ is vertex-decreasing. The event $ \cB_3$, which is the event that some defect is not covered by
enough $k$-sets is neither vertex-increasing nor vertex-decreasing as defined above. Here we make one small adjustment
to the variable: We consider covering $k$-sets that are restricted to the first $b_{k+1,j}$ vertices of the graph. This does not
materially change any estimates for $ n \in[ b_{k+1,j}, c_{k+1, j}]$ and yields a 
vertex-increasing event. Letting $ \tilde\cB_3$ be this alternate version of the event 
$ \cB_3$ we conclude that with high probability $ \cB = \cB_1 \vee \cB_2 \vee \tilde\cB_3 \vee \cB_4 $ does not hold
for any $n \in [ b_{k+1,r}, c_{k+1,r}]$. Then, following the proof of \cref{lem:new-upper} we conclude that
on the event $ \{ n \ge T_1\} \wedge \cB^c$ we have $ n \ge T_2$. Thus $ T_2 \le T_1$ with high probability.
\end{proof}

\appendix
\section{Proof of \texorpdfstring{\cref{prop:EKR-BS}}{Proposition 18}} \label{sec:far}

\begin{proof}[Proof of \cref{prop:EKR-BS}]
   Let $\cF$ be the set of graphs on $[m]$ that are $m^{3/5}$-far from being $r$-partite, and recall that we are to show $|\cF| \le 2^{{\rm ex}(m, K_{r+1})}  e^{ - m (\log m)^3/2}$. We begin by gathering some definitions and facts from \cite{balogh2019efficient}.  
    For each $ G \in \cF$ we define $ t(G)$ to be the minimum number of
    edges that need to be deleted from $G$ to make the graph $r$-partite.
    We sort $\cF$ into a number of subsets.
    Define
    \[
    \cF_{\rm far} = \left\{ G \in \cF: t(G) \ge m^2/ ( 8 \log m)^{15}  \right\} 
    \ \ \ \text{ and } \ \ \ \cF_{\rm close} = \cF \setminus \cF_{\rm far}.  \]
    For each $ G \in \cF_{\rm close}$ we define $ \Pi(G)$ to be an arbitrary
    $r$-partition of $U$ whose parts contain $ t(G)$ edges of $G$. We
    set that a $r$-partition $\Pi$ is {\em balanced} if each part contains
    at least $m/2r$ vertices. Define
    \[ \cF_{\rm close}^u = \{ G \in \cF_{\rm close}: \Pi(G) \text{ is unbalanced}\}\]
    The following bounds are established\footnote{These bounds are corollaries of the proofs of Theorem~6.2 and Lemma~6.5, respectively, from \cite{balogh2019efficient}. The bounds we give are stronger than the bounds in the statements of Theorem~6.2 and Lemma~6.5, but the improvements can be directly read off of the final lines off the proofs in \cite{balogh2019efficient}.} in \cite{balogh2019efficient}: for $r \leq \log m / (121 \log\log m)$ we have 
    \begin{equation}
     | \cF_{\rm far} |  \le 2^{ {\rm ex}(m, K_{r+1})} e^{ - \Omega\left( m^{2 - 1 /(8r)}   \right)}\,, \qquad  
     | \cF_{\rm close}^u | \le 2^{ {\rm ex}(m, K_{r+1})} e^{ - \Omega\left( m^{2}/r^2   \right)}\,. \label{eq:cF-far-close-bds}
    \end{equation}

    Now let $ \cP_b$ be the set of balanced $r$-partitions of $ U$. For each $ \Pi \in \cP_b$ and $ m^{3/5} < t < m^2/( 8 \log m)^{15}$ let
    \[ \cF_{t, \Pi} = \{ G \in \cF_{\rm close}: t(G) = t \text { and } \Pi(G) = \Pi  \}. \]
    Applying \eqref{eq:cF-far-close-bds} gives 
    \begin{equation}
        |\cF| \le | \cF_{\rm far}| + | \cF_{\rm close}^u| + \sum_{ \Pi \in \cP_b} \sum_{t= m^{3/5}}^{ m^2/ (8 \log m)^{15}}| \cF_{t, \Pi}| 
    \le 2^{{ \rm ex}(m, K_{r+1}) - m^{3/2}} + r^m \max_{\Pi}\sum_{t= m^{3/5}}^{ m^2/ (8 \log m)^{15}}| \cF_{t, \Pi}| . \label{eq:cF-first-bd}
    \end{equation} 
    
    In order to bound $ \cF_{t, \Pi}$ we note, following Balogh and Samotij, 
    that  a graph with $t$ edges has
    either a vertex of degree $D$ or a matching that consists of $t/D$ edges.  
    We use the following two Lemmas
    to bound the number of graphs in $ \cF_{t, \Pi}$ in these two categories. Recall that $ K_\Pi$ is the complete multipartite graph with parts given by $\Pi$.
    \begin{lemma}[Lemma~6.8 of \cite{balogh2019efficient}] \label{lem:star}
    Let $ D$ be an integer satisfying $D \ge 2^rr$. Suppose that $ \Pi$ is an $r$-partition of $ U$ and that $S$ is a copy of $ K_{1,D}$ with $ V(S)$ contained in some part $P$ of $ \Pi$. If $v \in P$ is the center  vertex of $S$ then
    \[ \left| \left\{ G \subseteq K_\Pi: G \cup S \not\supseteq K_{r+1} \text{ and }
    {\rm deg}_G(v,Q) \ge D \text{ for all } Q \in \Pi \setminus \{P\}
    \right\} \right| \le 2^{ e( K_\Pi) - \frac{D^2}{8r^2}} \]
    \end{lemma}
    \begin{lemma}[Lemma~6.9 of \cite{balogh2019efficient}] \label{lem:matching}
    Suppose that  $\Pi$ is a balanced $r$-partition of $U$ and that $M$ is a matching with $s$ edges such that $ V(M) \subseteq P$ for some $P \in \Pi$. If $ r^2 2^{r+3} \le m$ then
    \[ \left| \left\{ G \subseteq K_\Pi: G \cup M \not\supseteq K_{r+1}  \right\} \right| 
    \le 2^{ e(K_\Pi) - \frac{sm}{2^{10}r^4}}
    \]
    \end{lemma}
    \noindent
    We apply these Lemmas with $ D = (mt)^{1/3} $. A graph in $\cF_{t, \Pi}$ has either a vertex of degree $D$ in its own part, or else a matching of size at least $t/D$ whose edges all lie within the parts; some part has at least $s=t/Dr$ of these edges. Note that the additional condition on
    the degree of $v$ required for Lemma~\ref{lem:star} follows from the choice of $\Pi$ as a minimizer of the parameter $t(G)$. Indeed, if $v$ had fewer than $D$ neighbors in some other part $Q$, then we could move $v$ to $Q$ to make the graph closer to being $r$-partite. Note further, that Lemma~\ref{lem:star}~and~\ref{lem:matching} do not specify all edges in the graphs in question and so we have to take this into account in our estimate. Observe that for $t$ in the given range we have $ (mt)^{2/3}  > t ( \log m)^2, m(\log m)^4$. Consequently, 
    \[ | \cF_{t, \Pi}| \le 2^{ e( K_\Pi)} \left[ (m^2)^t 2^{- \frac{ D^2}{8r^2}} + (m^2)^t 2^{- \frac{ tm}{2^{10} r^5 D}}\right] \le 2^{ {\rm ex}(m, K_{r+1}) - m (\log m)^3}\,.\]
    Combining with \eqref{eq:cF-first-bd} completes the proof.
\end{proof}

\bibliography{bib.bib}
\bibliographystyle{abbrv}

\end{document}